\documentclass[12pt,reqno]{amsart}
\usepackage{amsmath,amsthm, amssymb}
\usepackage{mathrsfs}
\usepackage{graphicx}
\usepackage{mathtools}
\usepackage[update,prepend]{epstopdf}
\usepackage{epsfig}
\usepackage{color}
\usepackage{enumerate}

\usepackage[font=small]{caption}

\mathtoolsset{showonlyrefs=false}
\numberwithin{equation}{section}
		
\newcommand{\dd}{{{\rm d}}}
\newcommand{\ii}{{\rm i}}
\newcommand{\PT}{{\mathcal{PT}}}

\newcommand{\ie}{{\emph{i.e.}}}
\newcommand{\eg}{{\emph{e.g.}}}

\newcommand{\FNLS}{F_{\text{NLS}}}

\def\eps{\varepsilon}

\def\ri{{\rm i}}

\def\util{\tilde{u}}
\def\ubartil{\tilde{\overline{u}}}
\def\vtil{\tilde{v}}
\def\vbartil{\tilde{\overline{v}}}
\def\uhat{\hat{u}}

\newcommand{\R}{\mathbb{R}}
\newcommand{\C}{\mathbb{C}}

\newcommand{\B}{\mathbb{B}}

\newcommand{\N}{\mathbb{N}}

\newcommand{\Z}{\mathbb{Z}}

\newcommand{\cB}{{\mathcal B}}

\newcommand{\cF}{{\mathcal F}}

\newcommand{\cX}{{\mathcal X}}

\renewcommand{\phi}{\varphi}

\newcommand{\pa}{\partial}

\newcommand{\supp}{\text{\rm supp}}

\newcommand{\bspm}{\left(\begin{smallmatrix}}\newcommand{\espm}{\end{smallmatrix}\right)}
\newcommand{\bpm}{\begin{pmatrix}}\newcommand{\epm}{\end{pmatrix}}

\def\blem{\begin{lemma}}\def\elem{\end{lemma}}
\def\bthm{\begin{theorem}}\def\ethm{\end{theorem}}
\def\bprop{\begin{proposition}}\def\eprop{\end{proposition}}
\def\bcor{\begin{corollary}}\def\ecor{\end{corollary}}
\def\beq{\begin{equation}}\def\eeq{\end{equation}}
\def\bpf{\begin{proof}}\def\epf{\end{proof}}
\def\bex{\begin{example}}\def\eex{\end{example}}
\def\brem{\begin{remark}}\def\erem{\end{remark}}

\begin{document}

\theoremstyle{plain}
\newtheorem{define}{Definition}
\newtheorem{theorem}{Theorem}
\newtheorem{lemma}[theorem]{Lemma}
\newtheorem{criterion}[theorem]{Criterion}
\newtheorem{proposition}[theorem]{Proposition}
\newtheorem{corollary}[theorem]{Corollary}
\renewcommand{\proofname}{Proof}

\theoremstyle{definition}
\newtheorem{example}[theorem]{Example}
\newtheorem{remark}[theorem]{Remark}
\newtheorem{ass}{Assumption}
\renewcommand\theass{{\upshape{(\Roman{ass})}}}
%


\title[Bifurcation of nonlinear bound states with $\PT$-symmetry]{Bifurcation of nonlinear bound states in the periodic Gross-Pitaevskii equation with $\PT$-symmetry}

\author{Tom\'a\v{s} Dohnal}
\address[Tom\'a\v{s} Dohnal]{
Fachbereich Mathematik, Technical University Dortmund,
Vogelpothsweg 87, 44221 Dortmund, Germany}
\email{tomas.dohnal@math.tu-dortmund.de}

\author{Dmitry Pelinovsky}
\address[Dmitry Pelinovsky]{Department of Mathematics, McMaster University, Hamilton, Ontario, Canada, L8S 4K1}
\email{dmpeli@math.mcmaster.ca}
\address{Department of Applied Mathematics,
Nizhny Novgorod State Technical University, 24 Minin street, 603950 Nizhny Novgorod, Russia}



\date{\today}

\begin{abstract}
The stationary Gross--Pitaevskii equation in one dimension is considered with a complex periodic potential satisfying the conditions
of the $\PT$ (parity-time reversal) symmetry. Under rather general assumptions on the potentials we prove bifurcations
of $\PT$-symmetric nonlinear bound states from the end points of a real interval in the spectrum
of the non-selfadjoint linear Schr\"{o}dinger operator with a complex $\PT$-symmetric periodic potential.
The nonlinear bound states are approximated by the effective amplitude equation, which bears the form
of the cubic nonlinear Schr\"{o}dinger equation. In addition we provide sufficient conditions for the
appearance of complex spectral bands when the complex $\PT$-symmetric potential has an asymptotically small imaginary part.
\end{abstract}

\thanks{
The research was initiated during the LMS-Durham symposium on ``Mathematical and Computational Aspects of Maxwell's Equations'' in July 2016. The research of T.D. is partly supported by the \emph{German Research Foundation}, DFG grant No.  DO1467/3-1.
The research of D.P. is performed with financial support of the state task in the sphere of scientific activity of Russian Federation (Task No. 5.5176.2017/8.9).
}

\maketitle

\section{Introduction}\label{S:intro}

We consider the stationary Gross-Pitaevskii (GP) equation
\begin{equation}\label{E:nls}
-\frac{d^2 u}{dx^2} + V(x) u +\sigma(x)|u|^2u = \omega u, \quad x\in \R
\end{equation}
with complex $2\pi$-periodic potentials $V$ and $\sigma$ and with a real parameter $\omega$.
The periodic potentials satisfy the conditions of the $\PT$ (parity-time reversal) symmetry
given by
\begin{equation}
\label{PT-symmetry}
V(-x)=\overline{V(x)}, \quad \sigma(-x)=\overline{\sigma(x)}, \quad \text{ for all } x\in \R.
\end{equation}

\begin{ass}\label{ass:Vsig}
Assume $V \in L^{\infty}_{\rm per}(0,2\pi)$ and
 $\sigma\in H^s_\text{per}(0,2\pi)$ with $s > 1/2$
satisfy the $\PT$-symmetry condition (\ref{PT-symmetry}).
\end{ass}

Consider the linear Schr\"{o}dinger operator
\begin{equation}
\label{Schrodinger}
\mathcal{L} := -\frac{d^2}{dx^2}+V : \quad H^2(\mathbb{R}) \to L^2(\mathbb{R}),
\end{equation}
which is not self-adjoint if $V$ is complex. Nevertheless, we assume the existence of a real
spectral interval in the spectrum of $\mathcal{L}$ and prove the existence of $H^s(\R)$-solutions
(with $s>1/2$) to the stationary GP equation (\ref{E:nls}) bifurcating from an edge $\omega_* \in \R$ of the spectral interval.
We call these solutions nonlinear bound states. They correspond to standing waves $\psi(x,t)=e^{-\ri\omega t}u(x)$ of the
$t-$dependent Gross-Pitaevskii (GP) equation
\begin{equation}
\label{nls}
\ri \pa_t\psi = -\pa_x^2\psi + V(x)\psi + \sigma(x)|\psi|^2\psi.
\end{equation}
The bifurcating solutions $u$ are approximated via
a slowly varying envelope ansatz. In a generic case (non-vanishing second derivative
of the spectral function at the edge and non-vanishing coefficient in front of the
cubic nonlinear term) the effective envelope equation is a nonlinear Schr\"odinger equation with constant coefficients.

We work in the Sobolev space $H^s(\R)$ with $s>1/2$ in order to enjoy the algebra property
and the embedding of $H^s(\R)$ to the space of bounded and continuous functions decaying to zero at
infinity. Besides the Banach fixed point theorem the main analytical tool in the justification of the effective amplitude
equation is the Bloch transformation $\cB$ given formally by
\begin{equation}
\label{BT}
\util(x,k)=(\cB u)(x,k) = \sum_{n\in\Z} u(x+2\pi n) e^{-\ri k x - 2 \pi \ri n k}.
\end{equation}
The Bloch transformation was introduced by Gelfand \cite{Gelfand} and was used in the analysis of
the Schr\"odinger operator $\mathcal{L}$ with a real periodic potential $V$ \cite{RS}.
With $\B :=(-1/2,1/2]$ being the so-called Brillouin zone, the Bloch transform
$$\cB:H^s(\R) \to \cX_s :=L^2(\B,H^s(0,2\pi))$$
is an isomorphism for $s\geq 0$, see \cite{RS}, with the inverse given by
\begin{equation}
\label{IBT}
u(x)=(\cB^{-1}\util)(x)=\int_{\B}e^{\ri k x}\util(x,k)\dd k.
\end{equation}
As the norm in $\cX_s$ we choose
$$\|\util\|_{\cX_s}=\left(\int_\B \|\util(\cdot,k)\|^2_{H^s(0,2\pi)}\dd k\right)^{1/2}.$$

Under the Bloch transform the linear Schr\"{o}dinger operator (\ref{Schrodinger})
is represented by the family of linear operators parameterized by $k\in \B$ and given by
\begin{equation}
\label{operator-L}
L(k) := -\left(\frac{d}{dx}+\ii k\right)^2 +V : \quad H^2(0,2\pi)\to L^2(0,2\pi).
\end{equation}
Consider the family of Bloch eigenvalue problems
\beq\label{E:Bloch-ev}
\left\{ \begin{array}{l} L(k) p(\cdot,k) = \omega(k)p(\cdot,k), \\
p(x+2\pi,k)=p(x,k) \quad \quad \text{for all } x\in\R,
\end{array} \right.
\eeq
under the normalization condition $\|p(\cdot,k)\|_{L^2(0,2\pi)}=1$. In what follows we denote the eigenpairs of the Bloch eigenvalue problem \eqref{E:Bloch-ev} by $(\omega_m(k),p_m(x,k))$, $m\in \N$,
where the ordering can be done, \eg, according to the real part of the eigenvalues $\omega_m(k)$ (including their multiplicity).

We assume the existence of a real spectral interval given by an eigenvalue family $\{\omega_m(k)\}_{k\in \B}$ of
the periodic eigenvalue problem \eqref{E:Bloch-ev} for some $m \in \mathbb{N}$.
We also assume that this interval is disjoint from the rest of the spectrum of $\mathcal{L}$
given by (\ref{Schrodinger}) and that the end points of the spectral interval have
non-vanishing second derivative of $\omega_m$. In summary we pose the following.

\begin{ass}\label{ass:band}
For some $m\in \N$ the eigenvalue family $\omega_m$ is real with the spectral interval
\beq
\label{reality}
\omega_m(\B) = [a,b]\subset \R
\eeq
and with $[a,b]$ separated from the rest of the spectrum $\cup_{k\in \B}\sigma(L(k))$.
Moreover, the eigenvalue $\omega_m(k)$ is simple for each $k\in \B$.
For an end point $\omega_*\in \{a,b\}$ we assume that
\beq
\label{non-degeneracy}
\omega_m''(k_0)\neq 0,
\eeq
where $k_0 \in \B$ is the preimage of $\omega_*$ under the mapping $\omega_m$.
\end{ass}

In Fig. \ref{Fig:bd_str_gen} we plot the first (with respect to the real part) several eigenvalues of
the spectral problem \eqref{E:Bloch-ev} with
$$
V(x) = 2\cos(x)+\cos(2x)+\ri\gamma\sin(2x)
$$
for $\gamma=1$ (a) and $\gamma=3/2$ (b,c). They have been computed via a finite difference discretization.
For $\gamma=1$ all the lower spectral intervals appear real while for $\gamma=3/2$ a symmetry breaking has
occurred where the two lowest eigenvalue functions have collided and bifurcated into a complex conjugate pair.
The third spectral function remains real for $\gamma=3/2$ and its image is the marked interval $[a,b]$.
At $\omega=a$ we have $k_0=0$ and at $\omega=b$ it is $k_0=1/2$. The fourth and fifth spectral
functions also produce unstable eigenvalues in a small neighborhood of $k = 0$.
\begin{figure}[ht!]
\includegraphics[scale=0.5]{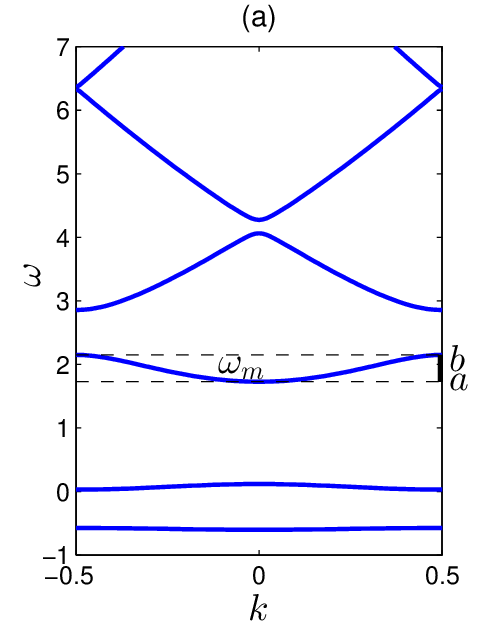}
\includegraphics[scale=0.5]{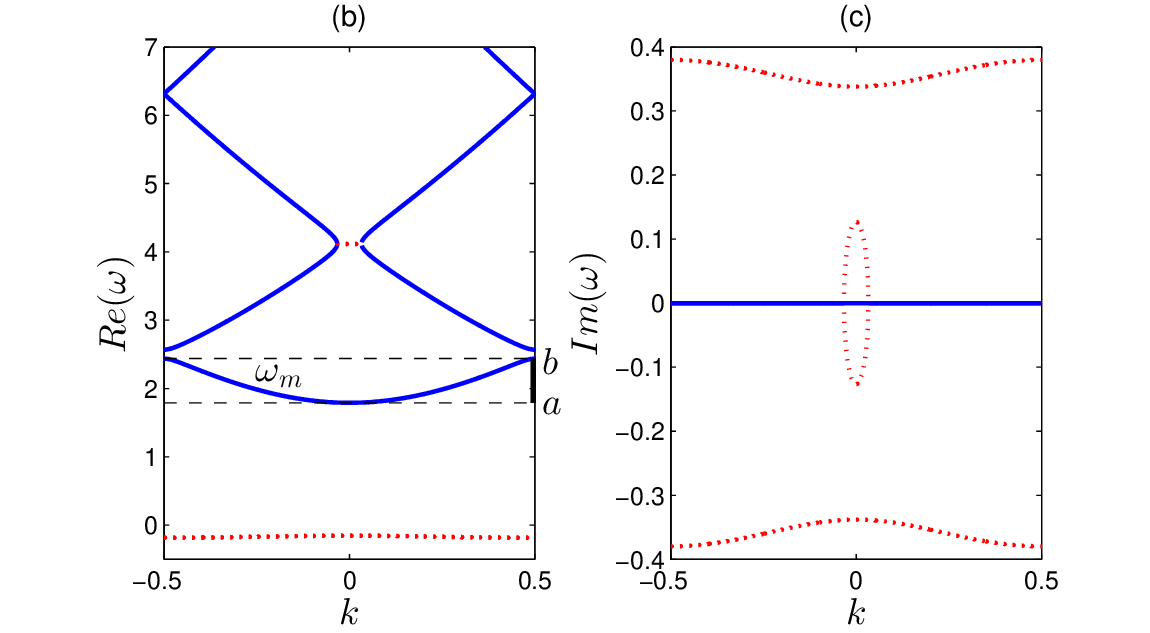}
\caption{The smallest (with respect to the real part) several eigenvalues $\omega_n(k)$ of the spectral problem \eqref{E:Bloch-ev}
with $V(x)=2\cos(x)+\cos(2x)+\ri\gamma \sin(2x)$, where $\gamma=1$ in (a) and $\gamma=3/2$ in (b) and (c). Purely real eigenvalues are plotted with the full blue line, complex eigenvalues are in dotted red. A real spectral interval $[a,b]=\omega_m(\B)$ is marked in both (a) and (b).}
\label{Fig:bd_str_gen}
\end{figure}

For real potentials $V \in L^{\infty}_{\rm per}((0,2\pi),\R)$,
the eigenvalue family $\{ \omega_m(k) \}_{k \in \B}$ cannot have an extremum for $k \notin \{0,1/2\}$
due to the symmetry $\omega_m(-k)=\omega_m(k)$, the $1-$periodicity of $\omega_m$ and the fact that
the differential equation $\mathcal{L} u = \lambda u$ posed for the Schr\"{o}dinger operator
(\ref{Schrodinger}) on the infinite line is of the second order \cite{RS}. Because the spectral interval
$[a,b]$ is isolated from the rest of the spectrum of $\mathcal{L}$,
the eigenvalue family $\{ \omega_m(k) \}_{k \in \B}$ then must have an extremum at either $k_0 = 0$
or $k_0 = 1/2$ and due to the smoothness  of simple eigenvalues with respect to parameters one has
$\omega_m'(k_0)=0$.

The extension of these properties to general non-self-adjoint operator $\mathcal{L}$
with complex potentials $V \in L^{\infty}_{\rm per}(0,2\pi)$ is not obvious. Nevertheless,
for $\PT$-symmetric potentials we show in Section 2 that the reflection symmetry
\beq
\label{symmetry-omega}
\omega_m(-k) = \omega_m(k) \quad \mbox{\rm for all \;} k \in \B,
\eeq
still holds for the eigenvalue family $\{ \omega_m(k) \}_{k \in \B}$ in Assumption \ref{ass:band}.
The $1-$periodicity and smoothness of $\omega_m$ in $k$ hold clearly as well.
Finally, the possibility of an extremum of the eigenvalue family $\{ \omega_m(k) \}_{k \in \B}$
at $k_0 \in (0,1/2)$ is excluded by the same argument as in the case of real potentials.
Indeed, if an extremum of $\omega_m$ exists at $k_0$, it also occurs at $-k_0$ by symmetry (\ref{symmetry-omega}).
Therefore, on one side of the extremal value of $\omega_m$, we have four bounded linearly independent
solutions of the eigenvalue problem $\mathcal{L} u = \lambda u$ on the infinite line,
which contradicts the fact that the eigenvalue problem is given by a second-order differential
equation. Hence, if the spectral band $[a,b]$ is isolated, then the eigenvalue family
$\{ \omega_m(k) \}_{k \in \B}$ has an extremum at either $k_0 = 0$ or $k_0 = 1/2$ and
\beq\label{E:om_stationary}
\omega_m'(k_0)=0 \ \text{ for } \ k_0=0 \; {\rm or} \; k_0 = \frac{1}{2}.
\eeq

For any eigenpair $(\omega(k),p(\cdot,k))$ of the spectral problem \eqref{E:Bloch-ev},
the pair $(\overline{\omega(k)},q(\cdot,k))$ with $q(x,k) := \overline{p(-x,k)}$ for all $x \in \mathbb{R}$
is also an eigenpair of the same eigenvalue problem.
This can be seen by complex conjugating $L(k) p(\cdot,k) = \omega(k) p(\cdot,k)$,
$$
\left[-\left(\frac{d}{dx}-\ii k\right)^2 +\overline{V(x)}\right] \overline{p(x,k)}= \overline{\omega(k)}~\overline{p(x,k)},
$$
using the $\PT$-symmetry (\ref{PT-symmetry}), and transforming $x\to -x$,
$$
\left[-\left(\frac{d}{dx}+\ii k\right)^2 +V(x)\right] \overline{p(-x,k)}= \overline{\omega(k)}~\overline{p(-x,k)},
$$
hence $L(k) q(\cdot,k) = \overline{\omega(k)} q(\cdot,k)$.
If $\omega(k)\in\R$ is simple for some $k \in \B$, then
$p(\cdot,k)$ and $q(\cdot,k)$ are linearly dependent and, thanks to the normalization condition,
the eigenfunction $p(\cdot,k)$ can be chosen to satisfy the $\PT$-symmetry condition,
\begin{equation}
\label{PT-symmetry-eigenfunction}
\overline{p(-x,k)}=p(x,k), \quad \mbox{\rm for all \;} x \in \mathbb{R}.
\end{equation}

In what follows, we say that the solution $u$ to the stationary GP equation \eqref{E:nls}
is $\PT$-symmetric if it satisfies the same $\PT$-symmetry condition,
\begin{equation}
\label{PT-symmetry-solution}
\overline{u(-x)} = u(x), \quad \mbox{\rm for all \;} x \in \mathbb{R}.
\end{equation}

We study the bifurcation of $\PT$-symmetric solutions to the stationary GP equation (\ref{E:nls})
from an endpoint $\omega_*$ of the real interval $[a,b]$ into a spectral gap. Hence,
we pick $\omega_*\in \{a,b\}$ as in Assumption \ref{ass:band} and set
\beq
\label{omega-decomposition}
\omega = \omega_*+\eps^2 \Omega,
\eeq
where $\eps$ is a formal small parameter and $\Omega = -1$ if $\omega_* = a$ or $\Omega = +1$ if $\omega_* = b$.

We prove in Section 3 that similarly to the case of real potentials $V$ and $\sigma$ \cite{DPS,DU}
(see also a review in Chapter 2 in \cite{Peli}), the family of nonlinear bound states in $H^s(\mathbb{R})$ with
$s > 1/2$ bifurcating from $\omega_*$ can be approximated via the slowly varying envelope ansatz
\beq\label{E:ans-formal}
u(x) \sim u_{\text{form}}(x) := \eps A(\eps x) e^{\ii k_0 x} p_m(x,k_0) \quad \text{ as } \quad \eps \to 0,
\eeq
where $A \in H^{s_A}(\R)$ with $s_A \geq 1$ satisfies the effective amplitude equation given by
the stationary nonlinear Schr\"odinger (NLS) equation,
\beq\label{E:SNLS}
- \frac{1}{2} \omega_m''(k_0) \frac{d^2A}{dX^2} + \Gamma |A|^2A = \Omega A,
\eeq
with
\beq
\label{Gamma}
\Gamma := \int_{-\pi}^{\pi}\sigma(x)p_m(x,k_0)|p_m(x,k_0)|^2\overline{p_m(-x,-k_0)}\dd x.
\eeq
The coefficient $\Gamma$ is real due to the $\PT$-symmetry of $\sigma$ and $p_m(\cdot,\pm k_0)$
in (\ref{PT-symmetry}) and (\ref{PT-symmetry-eigenfunction}).

If the effective equation \eqref{E:SNLS} has a bound state, we may expect the same for the GP equation \eqref{E:nls}.
It follows from the elementary phase-plane analysis that
bound states of the stationary NLS equation (\ref{E:SNLS}) exist if and only if
\beq\label{E:GS_iff}
\Gamma \neq 0 \ \text{ and } \ {\rm sign}(\Gamma) = -{\rm sign}(\omega''(k_0)) = {\rm sign}(\Omega).
\eeq
Real even bound states $A$ are unique and have an explicit ${\rm sech}$-function form, see Lemma 6.15 in  \cite{Fibich}. They
belong to $H^{s_A}(\R)$ for every $s_A \geq 0$. For the justification
of the effective equation we need the invertibility of the linearization operator of
the NLS equation at the bound state $A$. For this the translational and gauge invariances of
the differential equation (\ref{E:SNLS}) need to be eliminated, which is achieved if $A$ satisfies
the $\PT$-symmetry condition
\begin{equation}
\label{PT-symmetry-A}
\overline{A(-x)} = A(x), \quad \mbox{\rm for all \;} x \in \mathbb{R}.
\end{equation}

The following theorem justifies the effective amplitude equation (\ref{E:SNLS})
used for the approximation (\ref{E:ans-formal}) and constitutes the main result of this article.

\bthm
Let $1/2 < s \leq 2$ and $0 < r < 1/2$ and assume \ref{ass:Vsig} and \ref{ass:band}.
Let $A \in H^{s_A}(\R)$ be a $\PT$-symmetric solution to the stationary NLS equation \eqref{E:SNLS}
with $s_A\geq 1$ satisfying (\ref{PT-symmetry-A}).
Then there are constants $c > 0$ and $\eps_0>0$ such that for each $\eps \in (0,\eps_0)$ there exists
a $\PT$-symmetric solution $u\in H^s(\R)$ of the stationary GP equation \eqref{E:nls} with
$\omega=\omega_*+\eps^2\Omega$ satisfying (\ref{PT-symmetry-solution}) and
$$
\|u - u_{\text{form}}\|_{H^s(\R)}\leq c\eps^{r+1/2},
$$
where $u_{\text{form}}$ is defined in \eqref{E:ans-formal}.
\label{theorem-main}
\ethm

\begin{remark}
Condition \eqref{E:GS_iff} for the existence of NLS bound states implies that
the bifurcation is always into a spectral gap. At the lower spectral edge $\omega_*=a$,
where $\omega''(k_0)>0$, one has $\Omega<0$ and the bifurcation in $\omega$ is down from $a$;
analogously at the upper edge $b$, where $\omega''(k_0) < 0$, one has $\Omega > 0$
and the bifurcation in $\omega$ is up from $b$.
\end{remark}

\begin{remark}
Theorem \ref{theorem-main} guarantees that the error $u-u_\text{form}$ is indeed smaller than the approximation $u_\text{form}$ itself because $\|u_\text{form}\|_{H^s(\R)} \sim c\eps^{1/2}$ as $\eps \to 0$ for an $\eps$-independent $c$.
\end{remark}

\begin{remark}\label{R:gen}
The statement of Theorem \ref{theorem-main} can be generalized in a number of ways.
First, one can prove existence of smoother solutions with $s > 2$
provided that $p(\cdot,k_0)$ belongs to $H^s_{\rm per}(0,2\pi)$. This
would require a smoother potential $V$ than the one in \ref{ass:Vsig}.
Second, the spectral interval $[a,b]$ does not have to be real entirely, as in
\ref{ass:band}. For the justification result, it is sufficient that a little segment of
$[a,b]$ near the end point $\omega_*$ be real. Similarly, the simplicity assumption
of the eigenvalue $\omega_m(k)$ has to be satisfied only near the end point
that corresponds to $k = k_0$.
\end{remark}

\brem
Assumption \ref{ass:band} is not satisfied for an arbitrary complex $V$.
Propositions \ref{prop-1}, \ref{prop-2} and \ref{prop-3} in Section \ref{S:linspec} give sufficient conditions
for the occurrence of complex spectral bands
if $V(x)=U(x)+\ri \gamma W(x)$ with $U$ even and $W$ odd and with $\gamma >0$ arbitrarily small.
The sufficient conditions detect bifurcations of double eigenvalues at $\gamma=0$
into complex pairs of simple eigenvalues for $\gamma > 0$.
\erem

\begin{remark}
Recent interest in $\PT$-symmetric periodic potentials is explained by the experimental realization
of such optical lattices in physical experiments \cite{pt2,pt3}. Several computational
works were devoted to the existence and spectral stability of standing waves
in the GP equation with complex periodic potentials \cite{He,MMGC2008,nixon1} (see also the review in \cite{Konotop}).
Persistence of real spectrum in honeycomb $\PT$-symmetric potentials was studied in \cite{Curtis1}.
Small $\PT$-symmetric perturbations of honeycomb periodic potentials were considered in
\cite{Curtis2} and their effect on the nonlinear dynamics of the GP equation was studied.
A heuristic asymptotic method was used in \cite{nixon2} to approximate the standing
waves of the GP equation by ${\rm sech}$-solitons of the stationary NLS equation. Our work is the first one,
to the best of our knowledge, which gives a rigorous proof of the existence
of nonlinear bound states and their approximation by an effective equation for the bifurcation from
an edge of a real interval in the spectrum of a $\PT$-symmetric non-selfadjoint linear Schr\"{o}dinger operator.
\end{remark}

\begin{remark}
The bifurcation from simple eigenvalues is a more classical problem. The bifurcation of
nonlinear bound states from possibly complex eigenvalues of non-selfadjoint Fredholm operators
is covered in the pioneering works \cite{CR1971,IZE1976}. The bifurcation of nonlinear bound states
from simple real eigenvalues under an antilinear symmetry (which includes the $\PT$-symmetry)
has been shown for a large class of nonlinear problems in \cite{DS2016}. Earlier, in \cite{KPT2013}
this bifurcation was proved for the special case of a discrete NLS equation on a finite lattice.
The main difference between the bifurcation from a simple eigenvalue and from the edge of
a spectral interval is that in the former case the existence of the bifurcation is automatic
due to the separation of a simple eigenvalue from the rest of the spectrum while in the latter case
the edge is connected to the spectral band. In the case of simple eigenvalues a bifurcation occurs
even without symmetry assumptions and $\PT$-symmetry is used only to show
that the nonlinear bound state corresponds to real eigenvalue parameter.
In the case of a spectral \textit{interval} the symmetry is crucial for proving the bifurcation itself.
\end{remark}

The rest of the article is organized as follows.
Section \ref{S:adj_prob} covers the technical results associated
with the adjoint eigenvalue problem and with the Bloch transform.
Section \ref{S:NL_est} gives a proof of Theorem \ref{theorem-main}.
Section \ref{S:linspec} reports results based on perturbation theory which
give sufficient conditions on when Assumption \ref{ass:band} is not satisfied.

\section{The adjoint eigenvalue problem and the Bloch transform revisited}\label{S:adj_prob}

Since the spectral problem \eqref{E:Bloch-ev} is not self-adjoint in the presence
of complex periodic potentials, we also introduce the adjoint eigenvalue problem. By
the Fredholm theory (see Remark 6.23 in Chapter III.6.6 \cite{Kato}), eigenvalues of the
adjoint operator $L^*(k)$ are related to the eigenvalues of the operator $L(k)$
by complex conjugation. The adjoint eigenvalue problem is
written by
\beq\label{E:Bloch-ev-adj}
\left\{ \begin{array}{l}
L^*(k) p^*(\cdot,k) = \overline{\omega(k)} p^*(\cdot,k), \\
p^*(x+2\pi,k)=p^*(x,k) \quad \quad \text{for all } x\in\R,
\end{array} \right.
\eeq
where
\begin{equation}
\label{operator-L-adjoint}
L^*(k) := -\left(\frac{d}{dx}+\ii k\right)^2 + \overline{V} : \quad H^2(0,2\pi)\to L^2(0,2\pi)
\end{equation}
is the adjoint operator and $p^*(\cdot,k)$ is the adjoint eigenfunction.

If $\omega(k)$ is a simple eigenvalue of the spectral problem \eqref{E:Bloch-ev},
then $\overline{\omega(k)}$ is a simple eigenvalue of the adjoint spectral problem (\ref{E:Bloch-ev-adj})
and the adjoint eigenfunction can be uniquely normalized by $\langle p^*(\cdot,k), p(\cdot,k)\rangle=1$.
Indeed, if $\omega(k)$ is a simple eigenvalue, then $\langle p^*(\cdot,k), p(\cdot,k) \rangle = 0$ leads to
a contradiction. In detail, if $\langle p^*(\cdot,k), p(\cdot,k) \rangle = 0$, then we have
$$
p(\cdot,k) \in \text{Ker}(L(k)-\omega(k))\cap \text{Ran}(L(k)-\omega(k)).
$$
Therefore, there exists $\phi\in H^2(0,2\pi)\setminus\{0\}$ such that $(L(k)-\omega(k))\phi=p(\cdot,k)$.
At the same time, $\omega(k)$ being simple implies
$$
\text{Ker}(L(k)-\omega(k))^2=\text{Ker}(L(k)-\omega(k)),
$$
so that $\phi=cp(\cdot,k)$ with $c \in \R$, which contradicts
equation $(L(k)-\omega(k))\phi=p(\cdot,k)$. Thus, $\langle p^*(\cdot,k), p(\cdot,k) \rangle \neq 0$,
and the normalization $\langle p^*(\cdot,k), p(\cdot,k)\rangle=1$ can be used.

In the case of simple eigenvalues, the eigenpair $(\omega(k),p(\cdot,k))$
of the spectral problem \eqref{E:Bloch-ev} and the eigenpair $(\overline{\omega(k)},p^*(\cdot,k))$
of the adjoint problem \eqref{E:Bloch-ev-adj} are related via
\begin{equation}
\label{reflection}
\overline{\omega(k)} = \omega(-k), \quad p^*(x,k)=p(-x,-k).
\end{equation}
This follows from the $\PT$-symmetry of $V$ in (\ref{PT-symmetry}) such that
after the transformation $x\to -x$ and $k \to -k$, the adjoint problem $L^*(k) p^*(\cdot,k) = \overline{\omega(k)} p^*(\cdot,k)$
becomes
$$
\left[-\left(\frac{d}{dx}+\ii k\right)^2 +V(x)\right] p^*(-x,-k)= \overline{\omega(-k)} p^*(-x,-k),
$$
which coincides with the spectral problem (\ref{E:Bloch-ev}). As a result of the symmetry
reflection (\ref{reflection}) we obtain
\beq
\label{symmetry-omega-general}
\omega(-k) = \omega(k) \quad \mbox{\rm for all \;} k \in \B,
\eeq
for every simple real eigenvalue family $\{ \omega(k) \}_{k \in \B}$.
In addition, the eigenvalue family can be continued as a $1$-periodic function of $k$ on $\mathbb{R}$.
This symmetry and the $k-$smoothness of simple eigenvalues $\omega(k)$
justify \eqref{symmetry-omega} and \eqref{E:om_stationary} claimed in Section 1.

Before we proceed with the proof of Theorem \ref{theorem-main}, let us also elaborate properties
of the Bloch transform defined by (\ref{BT}) and (\ref{IBT}). Let
$\uhat$ be the standard Fourier transform of $u$ given by
$$
\hat{u}(\xi) = \left( \mathcal{F} u \right)(\xi) = \frac{1}{2\pi}\int_{\R} u(x)e^{-\ri \xi x}\dd x, \quad
u(x) = \left( \mathcal{F}^{-1} \hat{u} \right)(x) = \int_{\R}\hat{u}(\xi)e^{\ri \xi x}\dd \xi.
$$
Then, the Bloch transform (\ref{BT}) can also be related to the Fourier transform as follows:
\begin{equation}
\label{BT-Fourier}
\util(x,k) = \sum_{j \in\Z} e^{\ri j x} \hat{u}(k + j),
\end{equation}
see \cite{BSTU06ZAMP} or Section 2.1.2 in \cite{Peli}.

By construction of $\util$ in the definition of the Bloch transform in (\ref{BT}),
we have the continuation property for all $x \in \R$ and $k \in \R$:
\beq\label{E:util_per}
\util(x+2\pi,k)=\util(x,k) \quad \text{and} \quad \util(x,k+1)=e^{-\ri x}\util(x,k).
\eeq

For two functions $u,v\in H^s(\R)$ with $s>1/2$, the product $uv$ is also in $H^s(\R)$,
thanks to the Banach algebra of $H^s(\R)$ with respect to the pointwise multiplication \cite[Thm.4.39]{Adams-2003}.
In the Bloch space the multiplication operator is conjugate to the convolution operator:
$$
(\cB(uv))(x,k)=(\util *_\B \vtil)(x,k):=\int_\B \util(x,k-l)\vtil(x,l) \dd l
= \int_{\B+k_0} \util(x,k-l)\vtil(x,l) \dd l
$$
for any $k_0 \in \R$, where the last equality holds due to the $1-$quasi-periodicity of the Bloch transform
in the variable $k$. The convolution property follows from relation (\ref{BT-Fourier}). Note that due to the algebra property of $H^s(\R)$ for $s>1/2$
and the above identity we have also the algebra property
\beq\label{E:alg_Xs}
\|\util *_\B \vtil\|_{\cX_s}\leq c \|\util\|_{\cX_s} \|\vtil\|_{\cX_s} \quad \text{for any } \util,\vtil\in \cX_s \quad  \text{if } s>1/2,
\eeq
where the constant $c > 0$ depends on $s$.

Finally, for any $2\pi$-periodic and bounded function $\sigma$ we have the property
\begin{equation}
\label{Bloch-periodicity}
(\cB (\sigma u))(x,k) = \sigma(x) (\cB u)(x,k) \quad \mbox{\rm for all \;} x \in \mathbb{R} \;\; \mbox{\rm and} \;\; k \in \R.
\end{equation}
The commutativity property follows directly from the representation (\ref{BT}).

\section{Nonlinear estimates; proof of Theorem \ref{theorem-main}}\label{S:NL_est}

Problem \eqref{E:nls} transforms via the Bloch transform $\cB$ to the form
\beq\label{E:nls-Bloch}
(L(k)-\omega)\util(x,k)+\sigma(x)(\util *_\B \util *_\B \ubartil)(x,k)=0,
\eeq
where property (\ref{Bloch-periodicity}) has been used.

We decompose $\util$ into the part corresponding to the spectral band $\omega_m$ and the rest.
Note that we cannot use a full spectral decomposition of 	$\util(\cdot,k)$ as this is not available for non-selfadjoint problems.
For our decomposition we define the projections
$$
P_k : \; H^s(0,2\pi)\to \text{span}\{p_m(\cdot,k)\}
$$
and
$$
Q_k = I-P_k : \; H^s(0,2\pi)\to \text{span}\{p_m^*(\cdot,k)\}^{\perp}
$$
with $m\in \N$ fixed by assumption \ref{ass:band}, such that $(P_k \util)(\cdot,k) := \langle \util(\cdot,k),p_m^*(\cdot,k)\rangle p_m(\cdot,k)$,
where $\langle \cdot, \cdot\rangle$ is the standard $L^2(0,2\pi)-$inner product. Decomposing now the solution into
$$\util(x,k)=\util_1(x,k)+\util_2(x,k),$$
where
\begin{eqnarray*}
\util_1(\cdot,k) & = & P_k\util(\cdot,k) = U_1(k)p_m(\cdot,k) \quad \text{with } U_1(k)\in \C, \\
\util_2(\cdot,k) & = & Q_k\util(\cdot,k),
\end{eqnarray*}
and using $\omega = \omega_* + \eps^2 \Omega$ as is given by (\ref{omega-decomposition}),
equation \eqref{E:nls-Bloch} is written as a system of two equations given by
\begin{equation}
(\omega_m(k)-\omega_*-\eps^2\Omega)U_1(k) + \langle F(\util)(\cdot,k),p_m^*(\cdot,k)\rangle = 0 \label{E:P-eq}
\end{equation}
and
\begin{equation}
Q_k(L(k)-\omega_*-\eps^2\Omega)Q_k\util_2(x,k) + Q_k F(\util_1) + Q_k(F(\util)-F(\util_1)) = 0,
\label{E:Q-eq}
\end{equation}
where $F(\vtil)(x,k) := \sigma(x)(\vtil*_\B\vtil*_\B \vbartil)(x,k)$.
We note that $U_1$ is $1$-periodic because $\util(x,\cdot)$ and $p(x,\cdot)$ are
$1$-quasiperiodic.

Since $Q_k F(\util_1)$ in (\ref{E:Q-eq}) produces a large output, we need to
perform a near-identity transformation before we can proceed with the nonlinear estimates.
See the pioneering work \cite{KSM92} that explains this procedure.
We hence decompose $\util_2$ into
$$
\util_2(x,k)=\util_{2,1}(x,k)+\util_R(x,k),
$$
where $\util_{2,1}$ and $\util_R$ solve equations
\beq\label{E:util21-eq}
Q_k(L(k)-\omega_*-\eps^2\Omega)Q_k\util_{2,1}(x,k) + Q_k F(\util_1) = 0
\eeq
and
\beq\label{E:utilR-eq}
Q_k(L(k)-\omega_*-\eps^2\Omega)Q_k\util_R(x,k) + Q_k(F(\util)-F(\util_1)) = 0.
\eeq
The resulting system of equations is given by \eqref{E:P-eq}, \eqref{E:util21-eq} and \eqref{E:utilR-eq}.

The component $\util_1$ is supposed to approximately recover the Bloch transform of the formal ansatz \eqref{E:ans-formal}. Note that because $\cF \left(A(\eps \cdot)e^{\ri k_0 \cdot}\right)(k)=\eps^{-1} \hat{A}(\eps^{-1}(k-k_0))$, we have
\beq\label{E:uform_Bloch}
\cB(u_\text{form})(x,k)=\sum_{j\in \Z}\hat{A}\left(\frac{k-k_0+j}{\eps}\right)p_m(x,k_0)e^{\ri jx},
\eeq
where we have used properties (\ref{BT-Fourier}) and (\ref{Bloch-periodicity}).
Since $\hat{A}(\eps^{-1}(k-k_0))$ is concentrated near $k=k_0$, we decompose $U_1$ on $\B+k_0$
into a part compactly supported near $k_0$ and the rest. We write
\beq\label{E:U1-decomp}
U_1(k)=\hat{D}\left(\tfrac{k-k_0}{\eps}\right):=\hat{B}\left(\tfrac{k-k_0}{\eps}\right)+\hat{C}\left(\tfrac{k-k_0}{\eps}\right), \quad k\in \B+k_0,
\eeq
and continue $U_1$ outside $\B+k_0$ periodically with period one. For all $\eps$ small enough
the components $\hat{B}$ and $\hat{C}$ are defined by their support
$$
\begin{aligned}
&\supp\left(\hat{B}\left(\tfrac{\cdot-k_0}{\eps}\right)\right)\subset (k_0-\eps^{r},k_0+\eps^{r}),\\
&\supp \left(\hat{C}\left(\tfrac{\cdot-k_0}{\eps}\right)\right)  \subset (k_0+\B) \setminus (k_0-\eps^{r},k_0+\eps^{r}),
\end{aligned}
$$
where $r\in (0,1)$ is a parameter to be specified to suit the nonlinear estimates.
Equivalently, defining
$$
I_{\eps^{r-1}}:=(-\eps^{r-1},\eps^{r-1}),
$$
we have $\supp(\hat{B})\subset I_{\eps^{r-1}}$ and
$\supp(\hat{C}) \subset \eps^{-1}\B \setminus I_{\eps^{r-1}}$.
We note that neither $\hat{B}$, $\hat{C}$, nor $\hat{D}$ refer to the Fourier transform,
since they are defined in the Bloch space. On the other hand, $\hat{A}$ denotes the Fourier
transform of the amplitude variable $A$ that satisfies the effective amplitude equation
(\ref{E:SNLS}). In the Fourier variable $\kappa$ the amplitude $\hat{A}$ satisfies the
effective amplitude equation in the form
\beq\label{E:SNLS-Fourier}
\left( \frac{1}{2} \omega_m''(k_0) \kappa^2 - \Omega \right) \hat{A}(\kappa) + \Gamma \left(
\hat{A} * \hat{A} * \hat{\overline{A}} \right)(\kappa) = 0.
\eeq

We aim at constructing a solution $\util$ with $\hat{B}$ close to $\hat{A}$ on $I_{\eps^{r-1}}$
and with the other components $\hat{C}, \util_{2,1}$ and $\util_R$ being small corrections.
Hence,  due to \eqref{E:uform_Bloch} ansatz \eqref{E:U1-decomp} corresponds formally to the slowly varying envelope
ansatz \eqref{E:ans-formal}.

Obviously, system \eqref{E:P-eq}, \eqref{E:util21-eq} and \eqref{E:utilR-eq} is coupled
in the components $\hat{B},\hat{C}, \util_{2,1}$ and $\util_R$. Nevertheless, it can be approached by treating each problem independently with consistent assumptions on the form and size of the remaining components. In brief, our steps to construct such a solution
$\util\in \cX_s$ of \eqref{E:nls-Bloch}, i.e. of the original equation in the Bloch space, are as follows:
\begin{enumerate}
\item For any given $\util_1$ small, solve \eqref{E:util21-eq} uniquely to produce
a small $\util_{2,1}$ due to  the invertibility of $Q_k(L(k)-\omega_*-\eps^2\Omega)Q_k$ in $Q_k\cX_s$ for $\eps$ small enough.
\item For any given $\util_1$ small, apply the Banach fixed point theorem to \eqref{E:utilR-eq} in a neigborhood of zero
to find a small solution $\util_R$.
\item For any given $\hat{B}$ decaying sufficiently fast, apply the Banach fixed point theorem
to \eqref{E:P-eq} on the support of $\hat{C}$ to find a small $\hat{C}$.
\item Prove the existence of such solutions $\hat{B}$ to equation \eqref{E:P-eq} (with the component $\hat{C}$ given by step 3) on the support of $\hat{B}$ that are close to a solution $\hat{A}$ of equation \eqref{E:SNLS-Fourier}. It is in this step where a restriction to
the $\PT$-symmetric solutions is necessary. It allows for the invertibility of the Jacobian operator
at $\hat{A}$ associated with equation \eqref{E:SNLS-Fourier}.
\end{enumerate}
The rest of this section explains the details of each step in the justification analysis.
We denote a generic, positive, $\eps$-independent constant by $c$. It may change from one line to another line.
We also restrict our work to the space $\cX_s$ with $1/2 < s \leq 2$.

\subsection{Preliminary estimates}

We assume that for all $\eps > 0$ sufficiently small
 \beq\label{E:BC-est}
\|\hat{B}\|_{L^2_{s_B}(\R)} +  \|\hat{C}\|_{L^2_{s_C}(\R)} \leq c,
\eeq
where $s_B,s_C \geq 0$ are to be determined later and the space $L^2_s(\R)$ for $s\geq 0$ is
$$
L^2_s(\R):=\{f\in L^2(\R): \|f\|_{L^2_s(\R)}:=\|(1+|\cdot|)^sf\|_{L^2(\R)}<\infty\}.
$$
We estimate first $\|\util_1\|_{\cX_s}$.
Since $p_m(\cdot,k) \in H^2(0,2\pi)$ in the domain of $L(k)$ given by (\ref{operator-L}),
there is a positive constant $c$ such that for all $\eps > 0$ sufficiently small
and any $1/2 < s \leq 2$, we have
\beq\label{E:u1-est}
\|\util_1\|_{\cX_s} \leq \sup_{k\in \B}\|p_m(\cdot,k)\|_{H^2(0,2\pi)}\left\|\hat{D}\left(\frac{\cdot-k_0}{\eps}\right)\right\|_{L^2(\B)}
\leq c\eps^{1/2}\|\hat{D}\|_{L^2(\eps^{-1}\B)}.
\eeq

Next, let us consider the $\cX_s-$norm of $F(\util_1)$ given by
$$
F(\util_1) = \sigma(x) \int_\B\int_\B U_1(k-l)U_1(l-t)\overline{U}_1(-t)p_m(x,k-l)p_m(x,l-t)\overline{p_m(x,-t)}\dd t\dd l,
$$
which appears in equations \eqref{E:util21-eq} and \eqref{E:utilR-eq}. By assumption \ref{ass:Vsig} we get
\begin{eqnarray*}
\|F(\util_1)\|_{\cX_s} & \leq & \|\sigma\|_{H^s(0,2\pi)}\sup_{k\in \B}\|p_m(\cdot,k)\|_{H^2(0,2\pi)}^3
\|U_1*_\B U_1*_\B \overline{U}_1(-\cdot)\|_{L^2(\B)}\\
& \leq & c \|U_1*_\B U_1*_\B \overline{U}_1(-\cdot)\|_{L^2(\B)}.
\end{eqnarray*}
Next, we estimate
$$
\|U_1*_\B U_1*_\B \overline{U}_1(-\cdot)\|_{L^2(\B)} =
\eps^{5/2} \|\hat{B}*\hat{B}*\hat{\overline{B}}\|_{L^2(-3\eps^{r-1},3\eps^{r-1})}+\|E\|_{L^2(\B+k_0)},
$$
where
\begin{eqnarray}
\nonumber
E & := & 2W_1*_\B\hat{B}\left(\frac{\cdot-k_0}{\eps}\right)*_\B \hat{\overline{B}}\left(\frac{\cdot + k_0}{\eps}\right) + \hat{B}\left(\frac{\cdot-k_0}{\eps}\right)*_\B \hat{B}\left(\frac{\cdot-k_0}{\eps}\right)*_\B \overline{W_1}(-\cdot) \\
\label{E:E}
& \phantom{t}& \phantom{text} + \text{h.o.t.}
\end{eqnarray}
with
$$
W_1(k)=
\left\{
\begin{array}{ll}
\hat{C}\left(\frac{k-k_0}{\eps}\right), \quad & k-k_0 \in \B \\
U_1(k), \quad & k-k_0 \in \R \setminus \B
\end{array} \right.
$$
and with ``h.o.t.'' containing the remaining convolution terms, \ie~those quadratic and cubic in $W_1$.
A direct calculation yields
{\small $$
\begin{aligned}
&\left(W_1*_\B\hat{B}\left(\frac{\cdot-k_0}{\eps}\right)*_\B \hat{\overline{B}}\left(\frac{\cdot+k_0}{\eps}\right)\right)(k) = \\
& \sum_{n\in\{-1,0,1\}}\int\limits_{\stackrel{|l|<2\eps^r}{|k-l-k_0-n|<1/2}}\int\limits_{(-k_0-\eps^r,-k_0+\eps^r)\cap \B}\hat{C}\left(\frac{k-l-k_0-n}{\eps}\right)\hat{B}\left(\frac{l-t-k_0}{\eps}\right)\hat{\overline{B}}\left(\frac{t+k_0}{\eps}\right)\dd t \dd l,
\end{aligned}
$$}such that
$$
\left\|W_1*_\B\hat{B}\left(\frac{\cdot-k_0}{\eps}\right)*_\B
\hat{\overline{B}}\left(\frac{\cdot+k_0}{\eps}\right)\right\|_{L^2(\B+k_0)}
\leq c\eps^{5/2} \|\hat{C}\|_{L^2(\R)}\|\hat{B}\|_{L^1(\R)}^2,
$$
where we have used Young's inequality for convolutions. Next, we use the estimate
\beq\label{E:B_L1_L2s}
\|\hat{B}\|_{L^1(\R)}=\int_\R (1+|\kappa|)^{-s_B}(1+|\kappa|)^{s_B}|\hat{B}| \dd \kappa \leq c \|\hat{B}\|_{L^2_{s_B}(\R)},
\eeq
which holds for any $s_B>1/2$ because $\int_\R (1+|\kappa|)^{-2s_B} \dd k<\infty$ for $s_B>1/2$. Besides, due to the support of $\hat{C}$ we have
\beq\label{E:C_L2_L2s}
\|\hat{C}\|_{L^2(\R)} \leq \sup_{|\kappa|>\eps^{r-1}}(1+|\kappa|)^{-s_C}\|\hat{C}\|_{L^2_{s_C}(\R)}
\leq c\eps^{s_C(1-r)}\|\hat{C}\|_{L^2_{s_C}(\R)}
\eeq
for any $s_C>0$. With \eqref{E:B_L1_L2s} and \eqref{E:C_L2_L2s} we obtain
$$
\left\|W_1*_\B\hat{B}\left(\frac{\cdot-k_0}{\eps}\right)*_\B \hat{\overline{B}}\left(\frac{\cdot+k_0}{\eps}\right)\right\|_{L^2(\B+k_0)}\leq c\eps^{5/2+s_C(1-r)} \|\hat{C}\|_{L^2_{s_C}(\R)}\|\hat{B}\|_{L^2_{s_B}(\R)}^2
$$
for any $s_B > 1/2$ and $s_C > 0$.
Using similar computations for the higher-order terms in (\ref{E:E})
with the use of the estimate (\ref{E:B_L1_L2s}) for $\hat{C}$,
we obtain
\beq\label{E:E-est}
\|E\|_{L^2(\B+k_0)}\leq c\eps^{5/2+s_C(1-r)} \|\hat{C}\|_{L^2_{s_C}(\R)} \left(\|\hat{B}\|_{L^2_{s_B}(\R)}+\|\hat{C}\|_{L^2_{s_C}(\R)}\right)^2,
\eeq
for any $s_B > 1/2$ and $s_C > 1/2$. By using the estimate
$$
\|\hat{B}* \hat{B} * \hat{\overline{B}}\|_{L^2(-3\eps^{r-1},3\eps^{r-1})} \leq
\|\hat{B}\|_{L^2(\R)}\|\hat{B}\|^2_{L^1(\R)}\leq c\|\hat{B}\|_{L^2_{s_B}}^3,
$$
which follows from Young's inequality and estimate \eqref{E:B_L1_L2s},
we arrive at
\beq
\label{E:Fu1-est}
\|F(\util_1)\|_{\cX_s}\leq c\eps^{5/2}\left(\|\hat{B}\|_{L^2_{s_B}(\R)} + \|\hat{C}\|_{L^2_{s_C}(\R)} \right)^3
\eeq
for any $s_B>1/2$ and  $s_C > 1/2$.

\subsection{Component $\util_{2,1}$}

We solve now equation (\ref{E:util21-eq}) for $\util_{2,1}$ under assumption \eqref{E:BC-est}.
Recall that the operator
\begin{equation}
\label{operator-L-k}
M_k := Q_k(L(k)-\omega_*-\eps^2\Omega)Q_k,
\end{equation}
is invertible in $Q_k \cX_s$ with a bounded inverse. Thanks to estimate \eqref{E:Fu1-est}
there exists a unique solution $\util_{2,1}$ of (\ref{E:util21-eq}) which satisfies
\beq\label{E:u21-est}
\|\util_{2,1}\|_{\cX_s} \leq P \eps^{5/2},
\eeq
where $P$ depends polynomially on $\|\hat{B}\|_{L^2_{s_B}(\R)}$ and $\|\hat{C}\|_{L^2_{s_C}(\R)}$.

\subsection{Component $\util_R$}

Next, we solve equation \eqref{E:utilR-eq} for $\util_R$
via a Banach fixed point argument with $\util_1$ satisfying \eqref{E:BC-est} and $\util_{2,1}$ given as above.
We write
\beq
\label{component-u_R}
\util_R = M_k^{-1}Q_k \left( F(\util_1) - F(\util_1+\util_{2,1}+\util_R) \right) = :G(\util_R)
\eeq
We show the contraction property of $G$ in
$$
B^{\cX_s}_{K\eps^\eta}:=\{f\in \cX_s: \quad \|f\|_{\cX_s}\leq K\eps^{\eta}\}
$$
for some $K>0$ and $\eta>0$ to be determined. First, using \eqref{E:alg_Xs}, we estimate
{\small
$$
\begin{aligned}
\| G(\util_R) \|_{\cX_s} \leq &c \left[\|\util_1\|_{\cX_s}^2\left(\|\util_{2,1}\|_{\cX_s}+\|\util_R\|_{\cX_s}\right)
+\|\util_{2,1}\|_{\cX_s}^2\left(\|\util_{1}\|_{\cX_s}+\|\util_R\|_{\cX_s}\right)\right.\\
&\left.+
\|\util_R\|_{\cX_s}^2\left(\|\util_{2,1}\|_{\cX_s}+\|\util_1\|_{\cX_s}\right)+
\|\util_1\|_{\cX_s}\|\util_{2,1}\|_{\cX_s}\|\util_R\|_{\cX_s}+
\|\util_{2,1}\|_{\cX_s}^3+\|\util_R\|_{\cX_s}^3\right].
\end{aligned}
$$}Together with \eqref{E:u1-est} and \eqref{E:u21-est} we obtain for $\util_R\in B^{\cX_s}_{K\eps^\eta}$
\beq
\| G(\util_R)\|_{\cX_s} \leq P \left(
\eps^{7/2}+\eps^{\eta+1}+\eps^{2\eta+1/2}+\eps^{3\eta} \right),\label{E:Fdif-est}
\eeq
where $P$ depends polynomially on $\|\hat{B}\|_{L^2_{s_B}(\R)}$ and $\|\hat{C}\|_{L^2_{s_C}(\R)}$.
Clearly, if $5/2 \leq \eta\leq 7/2$, then $G: B^{\cX_s}_{K\eps^\eta}\to B^{\cX_s}_{K\eps^\eta}$ and
\beq
\label{E:Fdif-est1}
\|G(\util_R)\|_{\cX_s}\leq K\eps^{7/2},
\eeq
with $K$ dependent on $\|\hat{B}\|_{L^2_{s_B}(\R)}$ and $\|\hat{C}\|_{L^2_{s_C}(\R)}$.
We set $\eta=7/2$ for a balance.

For the contraction estimate, we consider $\util:=\util_1+\util_{2,1}+\util_R$ and $\vtil:=\util_1+\util_{2,1}+\vtil_R$.
A straightforward calculation leads to
\begin{eqnarray*}
\| G(\util_R)-G(\vtil_R)\|_{\cX_s} & \leq &
P \left(\|\util_1\|_{\cX_s}^2+\|\util_{2,1}\|_{\cX_s}^2+\|\util_R\|_{\cX_s}^2+\|\vtil_R\|_{\cX_s}^2\right)\|\util_R-\vtil_R\|_{\cX_s} \\
& \leq & \eps P \|\util_R-\vtil_R\|_{\cX_s},
\end{eqnarray*}
such that the contraction holds for $\eps>0$ small enough.

By the Banach fixed-point theorem, there exists a unique solution to
equation (\ref{component-u_R}) for $\util_R$, which satisfies the estimate
\beq
\label{E:uR-est}
\|\util_R\|_{\cX_s}\leq K\eps^{7/2},
\eeq
where $K$ depends polynomially on $\|\hat{B}\|_{L^2_{s_B}(\R)}$ and $\|\hat{C}\|_{L^2_{s_C}(\R)}$.

\subsection{Component $\hat{C}$}

Equation \eqref{E:P-eq} on the compact support of $\hat{C}$ can be rewritten as
\beq
\label{E:C-eq}
\hat{C}\left(\frac{k-k_0}{\eps}\right)=(\omega_* + \eps^2\Omega - \omega_m(k))^{-1}
\left( 1-\chi_{(-\eps^r,\eps^r)}(k-k_0) \right)
\langle F(\util)(\cdot,k),p_m^*(\cdot,k) \rangle
\eeq
with $k-k_0 \in \B\setminus (-\eps^r,\eps^r)$. Recall the decomposition
$\util_1(x,k)=\util_{1,B}(x,k)+\util_{1,C}(x,k)$, where for $k - k_0 \in \B$ we have
$$
\util_{1,B}(x,k) = \hat{B}\left(\frac{k-k_0}{\eps}\right)p_m(x,k), \quad
\util_{1,C}(x,k)=\hat{C}\left(\frac{k-k_0}{\eps}\right)p_m(x,k).
$$
Equation \eqref{E:C-eq} can be rewritten in the form
\beq\label{E:u1C-eq}
\util_{1,C}(x,k)= H(\util_{1,C})(x,k), \qquad k-k_0 \in \B\setminus (-\eps^r,\eps^r),
\eeq
where
$$
\begin{aligned}
H(\util_{1,C})(x,k):= \nu(k)&\left(1-\chi_{(-\eps^r,\eps^r)}(k-k_0)\right)\left[h^{-1}(k)\langle h(k)F(\util_{1,B})(\cdot,k),p_m^*(\cdot,k) \rangle \right.\\
&\left.+ \langle (F(\util)-F(\util_{1,B}))(\cdot,k),p_m^*(\cdot,k)\rangle\right] p_m(x,k),
\end{aligned}
$$
with $\nu(k):=(\omega_* + \eps^2\Omega - \omega_m(k))^{-1}$ and $h(k):=\left(1+ \eps^{-1} |k-k_0| \right)^{s_B}$
for some $s_B > 1/2$ to be determined. We note that
$$
\sup_{k-k_0\in \B\setminus (-\eps^r,\eps^r)}\nu(k)= c\eps^{-2r} \quad \mbox{\rm and} \quad
\sup_{k-k_0\in \B\setminus (-\eps^r,\eps^r)}h^{-1}(k)=c\eps^{(1-r)s_B}.
$$

The first term in $H(\util_{1,C})$ denoted as $T_1$ is estimated as follows
$$
\begin{aligned}
\|T_1\|_{\cX_s} &\leq c\eps^{(1-r)s_B-2r}\|hF(\util_{1,B})\|_{\cX_s}\\
&\leq c\eps^{(1-r)s_B-2r}
\left\| h \hat{B}\left(\frac{\cdot-k_0}{\eps}\right)*_\B \hat{B}\left(\frac{\cdot-k_0}{\eps}\right)*_\B \hat{\overline{B}}\left(\frac{\cdot+k_0}{\eps}\right) \right\|_{L^2(k_0-3\eps^r,k_0+3\eps^r)}\\
&\leq c\eps^{5/2+(1-r)s_B-2r}\left\|\hat{B}*\hat{B}*\hat{\overline{B}}\right\|_{L^2_{s_B}(-3\eps^{r-1},3\eps^{r-1})}\\
&\leq c\eps^{5/2+(1-r)s_B-2r}\|\hat{B}\|_{L^2_{s_B}(-3\eps^{r-1},3\eps^{r-1})}^3,
\end{aligned}
$$
if $s_B > 1/2$. The last inequality follows from $\hat{B}*\hat{B}*\hat{\bar{B}}=\widehat{|B|^2B}$,
the fact that the Fourier transform $\mathcal{F}$ is an isomorphism
$\mathcal{F}:H^s(\R)\to L^2_s(\R)$ for any $s\geq 0$ and from the algebra property of $H^s$ for $s>1/2$.

The second term in $H(\util_{1,C})$, denoted as $T_2$,
is estimated with the help of the algebra property \eqref{E:alg_Xs} of $\cX_s$ for $s>1/2$,
{\small $$
\begin{aligned}
\| T_2 \|_{\cX_s}\leq c\eps^{-2r} &\left[\|\util_{1,C}\|_{\cX_s}^3+\|\util_{1,C}\|_{\cX_s}^2(\|\util_{1,B}\|_{\cX_s}+\|\util_2\|_{\cX_s}) \right.\\
&\left. + \|\util_{1,C}\|_{\cX_s}(\|\util_{1,B}\|_{\cX_s}^2+\|\util_2\|_{\cX_s}^2) +\|\util_{1,B}\|_{\cX_s}^2\|\util_2\|_{\cX_s}+\|\util_{1,B}\|_{\cX_s}\|\util_2\|_{\cX_s}^2+\|\util_2\|_{\cX_s}^3
\right].
\end{aligned}
$$}We have the equivalence
\beq\label{E:u1B_B_equiv}
c_1 \eps^{1/2}\|\hat{B}\|_{L^2(\eps^{-1}\B)} \leq \|\util_{1,B}\|_{\cX_s} \leq c_2 \eps^{1/2}\|\hat{B}\|_{L^2(\eps^{-1}\B)}
\eeq
and
\beq\label{E:u1C_C_equiv}
c_1 \eps^{1/2}\|\hat{C}\|_{L^2(\eps^{-1}\B)} \leq \|\util_{1,C}\|_{\cX_s} \leq c_2 \eps^{1/2}\|\hat{C}\|_{L^2(\eps^{-1}\B)}
\eeq
for some $c_1,c_2>0$. Using \eqref{E:u21-est} and \eqref{E:uR-est}, we further have
$$
\begin{aligned}
\|\util_2\|_{\cX_s} \leq K \eps^{5/2},
\end{aligned}
$$
where $K$ depends polynomially on $\|\hat{B}\|_{L^2_{s_B}}$ and $\|\hat{C}\|_{L^2_{s_C}}$. As a result, we obtain
$$
\| T_2 \|_{\cX_s}\leq \rho \eps^{-2r}  \left( \|\util_{1,C}\|_{\cX_s}^3+\eps^{1/2}\|\util_{1,C}\|_{\cX_s}^2+ \eps\|\util_{1,C}\|_{\cX_s} +\eps^{7/2}\right) + \text{h.o.t.},
$$
where $\rho$ depends on $\|\hat{B}\|_{L^2_{s_B}}$ only, $\rho = \mathcal{O}(1)$ as $\eps\to 0$,
and where h.o.t. includes terms of higher order in $\eps$ or higher powers of $\|\util_{1,C}\|_{\cX_s}$.

Combining the estimates for $T_1$ and $T_2$, we obtain
\begin{eqnarray*}
\|H(\util_{1,C})\|_{\cX_s} & \leq & \rho \left( \eps^{5/2+(1-r)s_B-2r} + \eps^{7/2-2r} + \eps^{1-2r}\|\util_{1,C}\|_{\cX_s} +
\eps^{1/2-2r}\|\util_{1,C}\|_{\cX_s}^2 \right. \\ & \phantom{t} & \left.
+\eps^{-2r}\|\util_{1,C}\|_{\cX_s}^3\right) + \text{h.o.t.}
\end{eqnarray*}
Similarly one gets
\begin{eqnarray*}
\|H(\util_{1,C})-H(\vtil_{1,C})\|_{\cX_s} & \leq & \rho \left( \eps^{1-2r}+\eps^{1/2-2r}\left(\|\util_{1,C}\|_{\cX_s}+\|\vtil_{1,C}\|_{\cX_s}\right) \right. \\
& \phantom{t} & \left. +\eps^{-2r}\left(\|\util_{1,C}\|_{\cX_s}^2+\|\vtil_{1,C}\|_{\cX_s}^2\right)\right) \|\util_{1,C}-\vtil_{1,C}\|_{\cX_s}.
\end{eqnarray*}
Thus, the contraction holds for $\util_{1_C}$ in the ball
\beq\label{E:u1C-est}
\|\util_{1_C}\|_{\cX_s}\leq \rho \left( \eps^{5/2+(1-r)s_B-2r} + \eps^{7/2-2r} \right)
\eeq
if $r\in (0,1/2)$ and if $\eps>0$ is small enough.
By the Banach fixed-point theorem, there exists a unique $\util_{1_C}$
satisfying equation (\ref{E:u1C-eq}) and the bound (\ref{E:u1C-est}).

We also get an estimate for $\|\hat{C}\|_{L^2(\eps^{-1}\B)}$ by using \eqref{E:u1C_C_equiv},
\beq\label{E:C_est}
\|\hat{C}\|_{L^2(\eps^{-1} \B)} \leq \rho \left( \eps^{2+(1-r)s_B-2r} + \eps^{3-2r} \right),
\eeq
and an estimate of $\|\hat{C}\|_{L^2_{s_C}(\eps^{-1}\B)}$ by using
\beq\label{E:C_L2xi}
\|\hat{C}\|_{L^2_{s_C}(\eps^{-1}\B)}\leq \sup_{|\kappa|<\tfrac{1}{2\eps}}(1+|\kappa|)^{s_C}\|\hat{C}\|_{L^2(\eps^{-1}\B)}
\leq \rho \left( \eps^{2+(1-r)s_B-2r - s_C} + \eps^{3-2r - s_C} \right),
\eeq
where $\rho$ in both estimates \eqref{E:C_est} and \eqref{E:C_L2xi} depends polynomially on $\|\hat{B}\|_{L^2_{s_B}}$.

\subsection{Component $\hat{B}$}

Finally, we turn to the leading order component $\hat{B}(\tfrac{k-k_0}{\eps})p_m(x,k)$
of the solution $\util(x,k)$ and prove the existence of $\hat{B}$ close to $\hat{A}$,
a solution of the effective amplitude equation (\ref{E:SNLS-Fourier}).

Equation \eqref{E:P-eq} on the compact support of $\hat{B}$ can be rewritten as
\beq\label{E:B-eq}
(\omega_m(k)-\omega_*-\eps^2\Omega) \hat{B}\left(\frac{k-k_0}{\eps}\right) + \chi_{(-\eps^r,\eps^r)}(k-k_0)
\langle F(\util)(\cdot,k),p_m^*(\cdot,k) \rangle = 0
\eeq
with $k-k_0 \in(-\eps^r,\eps^r)$. Once again, we use a fixed point argument
to solve for $\hat{B}$. In order to close the procedure for constructing
$\util$, we need the constants in all estimates to depend only on norms of $\hat{B}$
and not on norms of $\hat{C}$. As all constants are polynomials
in $\|\hat{B}\|_{L^2_{s_B}}$ and $\|\hat{C}\|_{L^2_{s_C}}$,
we can employ \eqref{E:C_L2xi} to get rid of the dependence on $\|\hat{C}\|_{L^2_{s_C}}$.
We need to ensure, however,
$$
\min\{2+(1-r)s_B-2r-s_C,3-2r-s_C\}>0.
$$
Hence, $s_C > 1/2$ is further restricted by
$$
s_C<\min\{2+(1-r)s_B-2r,3-2r\}.
$$
Since $r<1/2$ and $s_B>1/2$, we can choose $s_C = 1 > 1/2$ to satisfy this restriction.

As we show below, the reduced bifurcation equation \eqref{E:B-eq} is a perturbation of
equation (\ref{E:SNLS-Fourier}).  To this end, we expand the band function
satisfying assumption \ref{ass:band} by
$$
\omega_m(k) = \omega_* + \frac{1}{2} \omega_m''(k_0) (k-k_0)^2 + \omega_r(k),
$$
where the remainder term satisfies the cubic estimate
\begin{equation}
\label{cubic-estimate}
|\omega_r(k)|\leq c|k-k_0|^3 \text{ for all  } k\in \B+k_0.
\end{equation}
By substituting this decomposition into equation \eqref{E:B-eq},
we can rewrite the problem for $\hat{B}$ in the form
\begin{eqnarray}
\nonumber
& \left( \frac{1}{2} \omega_m''(k_0) \left(\frac{k-k_0}{\eps}\right)^2 - \Omega \right) \hat{B}\left(\frac{k-k_0}{\eps}\right)
+ \eps^{-2}\chi_{(-\eps^r,\eps^r)}(k-k_0) \langle F(\util_1)(\cdot,k),p_m^*(\cdot,k) \rangle  \\
&
+ \eps^{-2}\chi_{(-\eps^r,\eps^r)}(k-k_0) \left[ \langle (F(\util)-F(\util_1))(\cdot,k),p_m^*(\cdot,k) \rangle - \omega_r(k)\hat{B}\left(\frac{k-k_0}{\eps}\right)\right] = 0.
\label{E:B-eq2}
\end{eqnarray}
The second term in equation (\ref{E:B-eq2}) recovers the nonlinearity coefficient in equation \eqref{E:SNLS-Fourier}.
Indeed, when we isolate the $\hat{B}$-component in $U_1$ and approximate all Bloch waves by those at $k=k_0$, we get
\begin{eqnarray*}
\langle F(\util_1)(\cdot,k),p_m^*(\cdot,k) \rangle & = & \Gamma\left(\hat{B}\left(\frac{\cdot-k_0}{\eps}\right)*_\B \hat{B}\left(\frac{\cdot-k_0}{\eps}\right)*_\B \hat{\overline{B}}\left(\frac{\cdot+k_0}{\eps}\right)\right)(k) \\
& \phantom{t} & \phantom{text} + \Gamma E(k) + H(k),
\end{eqnarray*}
where $\Gamma$ is given by (\ref{Gamma}), $E(k)$ is given in \eqref{E:E}, and
$$
H(k):=\int_\B\int_\B\left(\beta(k,k-l,l-t,t) - \Gamma\right)U_1(k-l)U_1(l-t)\overline{U_1}(-t)\dd t\dd l
$$
with
$$
\beta(k,k-l,l-t,t):= \langle\sigma p_m(\cdot,k-l)p_m(\cdot,l-t)\overline{p_m(\cdot,-t)},p_m^*(\cdot,k)\rangle.
$$
Note that $\Gamma = \beta(k_0,k_0,k_0,-k_0)$.

Next, we show the smallness of $E$, $H$ and the terms in the square brackets in \eqref{E:B-eq2}.
With the help of (\ref{cubic-estimate}) we obtain
\begin{eqnarray*}
\left\|\omega_r(k_0+\eps \cdot)\hat{B}\right\|^2_{L^2(\eps^{-1}\B)} &\leq &
c\eps^{6}\int_{-\eps^{r-1}}^{\eps^{r-1}}|\hat{B}(\kappa)|^2|\kappa|^6\dd\kappa\\
&\leq & c\eps^{6} \sup_{|\kappa|<\eps^{r-1}}\frac{|\kappa|^6}{(1+|\kappa|)^4}
\int_{-\eps^{r-1}}^{\eps^{r-1}}(1+|\kappa|)^{4}|\hat{B}(\kappa)|^2\dd\kappa \\
& \leq & c\eps^{4+2r}\|\hat{B}\|_{L^2_2(\R)}^2,
\end{eqnarray*}
so that
\beq
\label{estimate-omega-r}
\left\|\omega_r(k_0+\eps \cdot)\hat{B}\right\|_{L^2(\eps^{-1}\B)} \leq  c\eps^{2+r}\|\hat{B}\|_{L^2_2(\R)}.
\eeq
Estimate (\ref{estimate-omega-r}) dictates the choice of $s_B$, namely $s_B = 2 > 1/2$. Hence, from now on we work with
\[
s_B=2 \text{ and } s_C=1,
\]
in addition to $r < 1/2$.

To estimate $H$, we substitute the ansatz $U_1(k)=\hat{D}\left(\tfrac{k-k_0}{\eps}\right)$ for $k\in \B+k_0$ into $H(k)$
and use the transformations $t':=\eps^{-1}(t+k_0-m_3), l':=\eps^{-1}(l+m_2-m_3)$ and $\kappa:=\eps^{-1}(k-k_0)$. Then we get
\begin{eqnarray*}
H(k_0+\eps\kappa) = \eps^2 \sum_{m_{1,2,3}\in \{0,1\}}
\int\limits_{\eps^{-1}(\B+m_2-m_3)} \int\limits_{\eps^{-1} (\B+k_0-m_3)}  g \dd t'\dd l',
\end{eqnarray*}
where
\begin{eqnarray*}
g & = & \left(\beta(k_0+\eps \kappa,k_0+\eps(\kappa-l'),k_0+\eps(l'-t'),-k_0+\eps t') - \Gamma \right) \\
& \phantom{t} & \times \chi_{\eps^{-1}(\B-m_1-m_2+m_3)}(\kappa-l')\hat{D}\left(\kappa-l'+\eps^{-1}(m_1+m_2-m_3) \right)\\
& \phantom{t} & \times \chi_{\eps^{-1}\B}(l'-t')\hat{D}\left(l'-t'\right) \chi_{\eps^{-1}\B}(t')\hat{\overline{D}}\left(t'\right).
\end{eqnarray*}
Due to the analyticity of $k\mapsto p_m(\cdot,k)$ (recall that the eigenvalue family $\omega_m$ is simple) the coefficient $\beta$ is certainly Lipschitz continuous in each variable and we have
\begin{eqnarray*}
& \phantom{t} & \left|\beta(k_0+\eps \kappa,k_0+\eps(\kappa-l'),k_0+\eps(l'-t'),-k_0+\eps t')-\beta(k_0,k_0,k_0,-k_0)\right| \\
&\leq & c\eps(|\kappa|+|\kappa-l'|+|l'-t'|+|t'|) \\
& \leq & 2c\eps(|\kappa-l'|+|l'-t'|+|t'|).
\end{eqnarray*}
This leads to the estimate
$$
|H(k_0+\eps\kappa)|\leq c\eps^3 \left(2|h\hat{D}|*|\hat{D}|*|\hat{\overline{D}}|
+ |\hat{D}|*|\hat{D}|*|h\hat{\overline{D}}|\right)(\kappa) \ \text{with } h(\kappa):=\kappa,
$$
where $*$ is the convolution over the whole real line.
Applying Young's inequality for convolutions, we have
\begin{eqnarray*}
\|H(k_0+\eps \cdot)\|_{L^2(\eps^{-1}\B)} \leq c\eps^{3}\|\hat{D}\|_{L^2_1(\R)}\|\hat{D}\|_{L^1(\R)}^2
\leq c\eps^{3} \left( \|\hat{B}\|_{L^2_1(\R)} +\|\hat{C}\|_{L^2_1(\R)} \right)^3,
\end{eqnarray*}
where we have used estimate \eqref{E:B_L1_L2s} and the triangle inequality.
Employing now estimate \eqref{E:C_L2xi} with $r < 1/2$, $s_B = 2$, and $s_C = 1$, we finally obtain
\beq\label{E:H-est}
\|H(k_0+\eps \cdot)\|_{L^2(\eps^{-1}\B)} \leq \rho \eps^{3},
\eeq
where $\rho$ depends on $\|\hat{B}\|_{L^2_{s_B}}$ only.

To estimate $E$, we use \eqref{E:E-est} and (\ref{E:C_L2xi}) again and obtain for
$r < 1/2$, $s_B = 2$, and $s_C=1$
$$
\|E(k_0+\eps \cdot)\|_{L^2(\eps^{-1}\B)} \leq \rho \eps^{4},
$$
where $\rho$ depends on $\|\hat{B}\|_{L^2_{s_B}(\R)}$ only.

By using \eqref{E:Fdif-est1} and (\ref{E:C_L2xi}) again,
we obtain for $\delta F(k):= \langle \left( F(\util)-F(\util_1)\right)(\cdot,k),p_m^*(\cdot,k) \rangle$,
$$
\|\delta F(k_0+\eps \cdot)\|_{L^2(\eps^{-1}\B)}\leq \rho \eps^{3}
$$
where $\rho$ depends on $\|\hat{B}\|_{L^2_{s_B}(\R)}$ only.

Comparing estimates for $E$, $H$ and the terms in the square bracket in \eqref{E:B-eq2},
we conclude that the estimate (\ref{estimate-omega-r}) yields the leading order term.
In summary, equation \eqref{E:B-eq2} reads
\beq\label{E:Beq_R}
\left(\frac{1}{2} \omega_m''(k_0) \kappa^2-\Omega\right) \hat{B}(\kappa) +
\Gamma (\hat{B}*\hat{B}*\hat{\overline{B}})(\kappa) = \hat{R}(\hat{B})(\kappa),
\eeq
where $\kappa\in I_{\eps^{r-1}}:=(-\eps^{r-1},\eps^{r-1})$ and the remainder term satisfies
\beq\label{E:R-est}
\|\hat{R}(\hat{B})\|_{L^2(\eps^{-1}\B)}\leq \rho \eps^r,
\eeq
with $\rho$ depending on $\|\hat{B}\|_{L^2_{2}(\R)}$ only.

Equation \eqref{E:Beq_R} is a perturbed stationary NLS equation written in the Bloch
form on the compact support. In the following, we prove the existence of solutions
$\hat{B}$ to equation \eqref{E:Beq_R} close to $\chi_{I_{\eps^{r-1}}}\hat{A}$, where
$\hat{A}$ satisfies \eqref{E:SNLS-Fourier}.

We define
$$
F_{\text{NLS}}(\hat{B}):=\left(\frac{1}{2} \omega_m''(k_0) \kappa^2-\Omega\right) \hat{B}
+ \Gamma (\hat{B}*\hat{B}*\hat{\overline{B}})
$$
and write \eqref{E:Beq_R} as
$$
\FNLS(\hat{B})(\kappa)=\hat{R}(\hat{B})(\kappa), \quad \kappa \in I_{\eps^{r-1}},
$$
where $\supp(\hat{B})\subset I_{\eps^{r-1}}$. Letting
$$
\hat{B}=\hat{A}_\eps+\hat{b}, \text{  where } \hat{A}_\eps:= \chi_{I_{\eps^{r-1}}} \hat{A} \text{ and }\supp(\hat{b})\subset I_{\eps^{r-1}},
$$
we finally reformulate \eqref{E:Beq_R} as
\beq\label{E:b-eq}
\hat{J}_\eps \hat{b} = W(\hat{b}),
\eeq
where
\begin{eqnarray*}
\hat{J}_\eps(\kappa) & := & \chi_{I_{\eps^{r-1}}}(\kappa) D_{\hat{A}}\FNLS(\hat{A}_{\eps})(\kappa) \chi_{I_{\eps^{r-1}}}(\kappa), \\
W(\hat{b})(\kappa) & := & \chi_{I_{\eps^{r-1}}}(\kappa) \hat{R}(\hat{A}_\eps +\hat{b})(\kappa) -
\chi_{I_{\eps^{r-1}}}(\kappa) \left(\FNLS(\hat{A}_\eps+\hat{b})-\hat{J}_\eps\hat{b}\right)(\kappa).
\end{eqnarray*}
Here $D_{\hat{A}}\FNLS(\hat{A})$ is a symbolic notation for the Jacobian of $\FNLS$. Note that
$\FNLS$ is not complex differentiable but it is differentiable in real variables (after isolating the real and imaginary parts).

The Taylor expansion yields
$$
\FNLS(\hat{A}_\eps+\hat{b}) - \hat{J}_\eps\hat{b} = \FNLS(\hat{A}_\eps) +
(D_{\hat{A}}\FNLS(\hat{A}_\eps)-\hat{J}_\eps)\hat{b}+Q(\hat{b}),
$$
where $Q$ is quadratic in $\hat{b}$. The term $\FNLS(\hat{A}_\eps)$ does not vanish exactly
due to the convolution structure of the nonlinearity but we have
$$
\FNLS(\hat{A}_\eps) = \Gamma\left(\hat{A}_\eps*\hat{A}_\eps*\hat{\overline{A}}_\eps - \hat{A}*\hat{A}*\hat{\overline{A}}\right),
$$
where the right-hand-side includes terms of the form
$$
\hat{a}_\eps*\hat{A}_\eps*\hat{\overline{A}}_\eps, \text{ where }
\hat{a}_\eps:=\hat{A}-\hat{A}_\eps=(1-\chi_{I_{\eps^{r-1}}})\hat{A}
$$
and terms quadratic and cubic in $\hat{a}_\eps$. Similarly to estimates of $\hat{C}$,
we have
\begin{eqnarray*}
|\hat{a}_\eps(\kappa)| & \leq & (1+|\kappa|)^{s_A} |\hat{a}_\eps(\kappa)|\sup_{|\kappa|>\eps^{r-1}}(1+|\kappa|)^{-s_A}\\
 & \leq & c\eps^{s_A(1-r)}(1+|\kappa|)^{s_A} |\hat{a}_\eps(\kappa)|,
\end{eqnarray*}
where $s_A > 0$ is to be specified. By using Young's inequality, we obtain
\begin{eqnarray*}
\|\hat{a}_\eps*\hat{A}_\eps*\hat{\overline{A}}_\eps\|_{L^2(\R)} & \leq &
c\eps^{s_A(1-r)} \|(1+|\cdot|)^{s_A} \hat{A}\|_{L^2(\R)}\|\hat{A}_\eps\|_{L^1(\R)}^2 \\
& \leq & c\eps^{s_A(1-r)} \|\hat{A}\|_{L^2_{s_A}(\R)}^3,
\end{eqnarray*}
where the last estimate holds if $s_A>1/2$, see \eqref{E:B_L1_L2s}.
Thus, we have
\beq\label{E:FNLS-est}
\|\FNLS(\hat{A}_\eps)\|_{L^2(\R)}\leq c \eps^{s_A(1-r)} \|\hat{A}\|_{L^2_{s_A}(\R)}^3.
\eeq
Similarly, we obtain
\beq\label{E:JNLS-est}
\|(\chi_{I_{\eps^{r-1}}} D_{\hat{A}}\FNLS(\hat{A}_\eps)-\hat{J}_\eps) \hat{b} \|_{L^2(\R)}
\leq c\eps^{s_A(1-r)}\|\hat{A}\|_{L^2_{s_A}(\R)}^2\|\hat{b}\|_{L^2_{s_B}(\R)}
\eeq
for any $s_A>1/2.$

Combining \eqref{E:R-est}, \eqref{E:FNLS-est}, and \eqref{E:JNLS-est}, we have
$$
\|W(\hat{b})\|_{L^2(\R)} \leq c_A \left(\eps^r+\eps^{s_A(1-r)}+(\eps^r+\eps^{s_A(1-r)})\|\hat{b}\|_{L^2_2(\R)}+\|\hat{b}\|_{L^2_2(\R)}^2+\|\hat{b}\|_{L^2_2(\R)}^3\right),
$$
where $c_A$ depends on $\|\hat{A}\|_{L^2_{s_A}}$ only. Clearly, if $s_A\geq r/(1-r) \geq 1$ (so that $s_A \geq 1$ is used from now on),
then there exist positive constants $c_1$ and $c_2$ that only depend on $\|\hat{A}\|_{L^2_{s_A}}$ such that for all $\eps>0$ small enough
\beq
\label{fixed-1}
W:B^{2,2}_{c_1\eps^r}\to B^{2,0}_{c_1\eps^r},
\eeq
and
\beq
\label{fixed-2}
\|W(\hat{b}_1)-W(\hat{b}_2)\|_{L^2(\R)} \leq c_2\eps^r\|\hat{b}_1-\hat{b}_2\|_{L_2^2(\eps^{-1}\B)} \text{ for all } \hat{b}_1,\hat{b}_2 \in  B^{2,2}_{c_1\eps^r},
\eeq
where
\beq
\label{fixed-3}
B^{2,2}_{c_1\eps^r}:=\{\hat{b}\in L_2^2(\eps^{-1}\B): \|\hat{b}\|_{L_2^2(\eps^{-1}\B)} \leq c_1\eps^r\}.
\eeq

We wish to solve \eqref{E:b-eq} in $B^{2,2}_{c_1\eps^r}$ via the Banach fixed point iteration by writing
$\hat{b}=\hat{J}_\eps^{-1} W(\hat{b})$, where
$$
\hat{J}_\eps^{-1} : \quad L^2(\eps^{-1}\B) \to L^2_2(\eps^{-1}\B).
$$
The operator $\hat{J}_\eps^{-1}$ is, however, not bounded uniformly in $\eps$ (in a neighborhood of $\eps=0$)
because the Jacobian $\hat{J}_0 :=D_{\hat{A}}F_\text{NLS}(\hat{A})$ has a nontrivial kernel due to the shift and phase invariances of the
stationary NLS equation (\ref{E:SNLS}).

Indeed, as is well known (see Chapter 4 in \cite{Peli}),
the Jacobian $J_0$ at a bound state $A$ of the stationary NLS equation (\ref{E:SNLS}) is a diagonal operator of
two Schr\"{o}dinger operators
$$
L_+ := -\frac{1}{2} \omega_m''(k_0) \frac{d^2}{d X^2} - \Omega + 3 \Gamma |A(X)|^2 : \quad H^2(\R) \to L^2(\R)
$$
and
$$
L_- := -\frac{1}{2} \omega_m''(k_0) \frac{d^2}{d X^2} - \Omega + \Gamma |A(X)|^2 : \quad H^2(\R) \to L^2(\R)
$$
which act on the real and imaginary parts of the perturbation to $A$.
By the shift and phase invariances, both operators have kernels, namely
$$
{\rm Ker}(L_+) = {\rm span}\left(\frac{dA}{dX}\right) \quad \mbox{\rm and} \quad
{\rm Ker}(L_-) = {\rm span}(A),
$$
and the simple zero eigenvalue of $L_+$ and $L_-$ is isolated from the rest
of their spectra. By using the Fourier transform and the dualism between $H^2(\R)$ and $L^2_2(\R)$ spaces,
these facts imply that if $A \neq 0$ is a bound state to equation (\ref{E:SNLS}) in $H^{s_A}(\R)$ with $s_A \geq 1$,
then the kernel of $\hat{J}_0 : L^2_2(\R) \to L^2(\R)$ is two-dimensional and the double zero eigenvalue
is bounded away from the rest of the spectrum of $\hat{J}_0$. In a suitably selected subspace defined by a symmetry of the
stationary NLS equation (\ref{E:SNLS}), such that the invariances do not hold within this subspace,
the Jacobian $\hat{J}_0$ is invertible. The two invariances are avoided if we restrict to $\PT$-symmetric $A$ and $b$, \ie
$$
A(-x)=\overline{A(x)}, \quad b(-x)=\overline{b(x)},
$$
or equivalently
$$
\hat{A}(\kappa), \hat{b}(\kappa)\in \R \text{ for all }\kappa\in \R.
$$

Hence, for any given $\PT$-symmetric solution $A$ to equation \eqref{E:SNLS},
we consider a solution to the fixed-point equation \eqref{E:b-eq} in
$B^{2,2}_{c_1 \eps^r}$ for real $\hat{b}$. By the $\PT$-symmetry of the
original problem (\ref{E:nls}), all components of the decomposition of $u$
inherit the $\PT$-symmetry if $\hat{b}$ is real, so that
if $\hat{b}$ is real, then $\hat{J}_\eps^{-1}W(\hat{b})$ is real,
and the fixed-point equation \eqref{E:b-eq} is closed
in the space of $\PT$-symmetric solutions. Then, thanks to (\ref{fixed-1}), (\ref{fixed-2}), and
(\ref{fixed-3}), there exists a unique real solution $\hat{b} \in B^{2,2}_{c_1 \eps^r}$
of the fixed-point equation (\ref{E:b-eq}).

In order to understand the above inheritance property in more detail,
note that $u$ is $\PT$-symmetric if and only of $\util(\cdot,k)$ is $\PT-$symmetric for all $k\in\B$.
Hence, we can check the inheritance in the Bloch variable $\util$.
Clearly, we need to only check that $M_k$ and $F$ commute with the $\PT-$symmetry,
where $M_k$ is given by (\ref{operator-L-k}). We write
$$
\PT(\util)(x,k):=\overline{\util}(-x,k).
$$
For $M_k$ first note that $Q_k\PT=\PT Q_k$ because the eigenfunctions $p_m(\cdot,k)$ and $p_m^*(\cdot,k)$ are $\PT$-symmetric, such that
$$
\begin{aligned}
Q_k\PT(\util)(x,k)&=\PT(\util)(x,k)-\langle\PT(\util)(\cdot,k),p_m^*(\cdot,k)\rangle p_m(x,k)\\
&=\PT(\util)(x,k)-\langle\PT(\util)(\cdot,k),\PT(p_m^*)(\cdot,k)\rangle \PT (p_m)(x,k)\\
&=\PT(\util)(x,k)-\overline{\langle\util(\cdot,k),p_m^*(\cdot,k)\rangle} \PT (p_m)(x,k)\\
&=\PT (Q_k\util)(x,k).
\end{aligned}
$$
Due to the $\PT$-symmetry of $V$ we get also $L(k)\PT(\util) = \PT (L(k)\util).$ As a result $M_k\PT=\PT M_k$. Similarly, due to the $\PT-$symmetry of $\sigma$ we get $F(\PT(\util))=\PT(F(\util))$.

Therefore, for each component of $\util$ the Banach fixed point argument can be carried out in the $\PT$-symmetric subspace. All the resulting components of the decomposition are $\PT$-symmetric and hence the full solution $u$ is $\PT$-symmetric.

\subsection{Difference between the formal ansatz and $\util_{1,B}$}

To prove the inequality in Theorem \ref{theorem-main}, it remains to estimate the difference $u_{\text{form}}-\util_{1,B}$.
The formal ansatz $u_\text{form}$ in \eqref{E:ans-formal} translates in Bloch variables to the decomposition (\ref{E:uform_Bloch}).
We seek now an estimate of $\|\util_{\text{form}}-\util_{1,B}\|_{\cX_s}$, where  for all $k-k_0 \in \B$ we have
$$
\begin{aligned}
\util_{1,B}(x,k)=\hat{B}\left(\tfrac{k-k_0}{\eps}\right)p_m(x,k) = \left(\chi_{(-\eps^r,\eps^r)}(k-k_0)\hat{A}\left(\tfrac{k-k_0}{\eps}\right)+\hat{b}\left(\tfrac{k-k_0}{\eps}\right)\right)p_m(x,k).
\end{aligned}
$$
The estimate is carried out as follows:
\begin{eqnarray*}
\|\util_{\text{form}}-\util_{1,B}\|_{\cX_s}^2 & \leq & c\eps \|\hat{b}\|_{L^2(\eps^{-1}\B)}^2+\int_{k_0-\eps^r}^{k_0+\eps^r}\left|\hat{A}\left(\tfrac{k-k_0}{\eps}\right)\right|^2 \|p_m(\cdot,k)-p_m(\cdot,k_0)\|_{H^s(0,2\pi)}^2\dd k\\
& \phantom{t} & + \int_{\B+k_0}(1-\chi_{k_0-\eps^r,k_0+\eps^r}(k))\left|\hat{A}\left(\tfrac{k-k_0}{\eps}\right)\right|^2\dd k \|p_m(\cdot,k_0)\|_{H^s(0,2\pi)}^2\\
& \phantom{t} & +\sum_{j\in \Z\setminus\{0\}}\int_{\B+k_0}
\left|\hat{A}\left(\tfrac{k-k_0+j}{\eps}\right)\right|^2\dd k \|p_m(\cdot,k_0)e^{\ii j\cdot}\|_{H^s(0,2\pi)}^2,
\end{eqnarray*}
from which we obtain
\begin{eqnarray*}
\|\util_{\text{form}}-\util_{1,B}\|_{\cX_s}^2 & \leq & c\eps\left(\|\hat{b}\|_{L^2(\eps^{-1}\B)}^2+\eps^2
\int\limits_{|\kappa|<\eps^{r-1}}|\kappa|^2|\hat{A}(\kappa)|^2\dd k \right. \\
& \phantom{t} & \left. +\sup_{|\kappa|>\eps^{r-1}}(1+|\kappa|)^{-2s_A}\int\limits_{|\kappa|>\eps^{r-1}}(1+|\kappa|)^{2s_A}|\hat{A}(\kappa)|^2\dd \kappa\right.\\
& \phantom{t} & \left.+\sum_{j\in \Z\setminus\{0\}}\sup_{\kappa\in \eps^{-1}(\B+j)}(1+|\kappa|)^{-2s_A}\int\limits_{\kappa\in \eps^{-1}(\B+j)}(1+|\kappa|)^{2s_A}|\hat{A}(\kappa)|^2\dd \kappa\right)
\end{eqnarray*}
and hence
\begin{eqnarray*}
\|\util_{\text{form}}-\util_{1,B}\|_{\cX_s}^2 \leq c\eps \left(\|\hat{b}\|_{L^2(\eps^{-1}\B)}^2+\eps^2\|\hat{A}\|_{L^2_1(\R)}^2 +\eps^{2s_A(1-r)}\|\hat{A}\|_{L^2_{s_A}(\R)}^2\right).
\end{eqnarray*}
Together with $\|\hat{b}\|_{L_2^2(\eps^{-1}\B)} \leq c_1\eps^r$ (recall that $\hat{b}\in \B^{2,2}_{c_1\eps^r}$), this
estimate yields
\beq
\label{E:u1B-err}
\|\util_{\text{form}}-\util_{1,B}\|_{\cX_s} \leq c_A \eps^{r+1/2},
\eeq
where the constant $c_A$ depends polynomially on $\|\hat{A}\|_{L^2_{s_A}(\R)}$ with $s_A\geq 1$.

Estimate \eqref{E:u1B-err} together with \eqref{E:u21-est}, \eqref{E:uR-est} and \eqref{E:u1C-est}
and the triangle inequality complete the proof of Theorem \ref{theorem-main}.

\section{The spectral assumption revisited}\label{S:linspec}

The proof of Theorem \ref{theorem-main} relies critically  on Assumption \ref{ass:band}
that the spectral band $[a,b]$ is real and isolated from the rest of the spectrum
of the Schr\"{o}dinger operator $\mathcal{L}$ in (\ref{Schrodinger}) (although this can be generalized as explained in Remark \ref{R:gen} in Section \ref{S:intro}).
In real periodic potentials, every spectral band is real but the two bands may touch
at a point with no spectral gap.

For our purposes we say that the point $(k_0,\mu)\in \B\times \R$ is a \textit{Dirac point} in the one-dimensional case if
two eigenvalue families $k\mapsto \omega_m(k)$ and $k\mapsto\omega_{m+1}(k)$ of
the spectral problem \eqref{E:Bloch-ev} are real on some neighborhood around $k=k_0$, if $\omega_m(k_0)=\omega_{m+1}(k_0)=\mu$ and if $\omega_m$ and $\omega_{m+1}$ are not differentiable at $k_0$.

Due to the Lipschitz continuity of all $\omega_m$ (as follows, e.g., by a direct modification of the proof for the Helmholtz equation in \cite{CV1997}), a Dirac point is where $\omega_m$ and $\omega_{m+1}$ are conical in shape. Moreover, at a Dirac point $(k_0,\mu)$  two linearly independent eigenvectors of the spectral problem (\ref{E:Bloch-ev})
exist and a system of two stationary nonlinear Dirac-type equations can be derived and justified with an analogous analysis as in the case of the stationary NLS equations \cite{PelSch} (see also Chapter 2 in \cite{Peli}).

In $\PT$-symmetric periodic potentials
with the honeycomb symmetry in two spatial dimensions, a necessary and sufficient condition was derived in \cite{Curtis1}
at the Dirac point by the perturbation theory that shows when the spectral bands remain real
under a complex-valued perturbation.

Here we iterate the same question for the $\PT$-symmetric potential $V$ in Assumption \ref{ass:Vsig}.
We derive perturbative results related to splitting of Dirac points, when the real periodic potential
is perturbed by a purely imaginary perturbation potential. Therefore, we represent
\begin{equation}
\label{complex-potential}
V(x) = U(x) + \ri \gamma W(x),
\end{equation}
where $\gamma \in \mathbb{R}$ is the perturbation parameter and the real potentials
$U, W \in L^{\infty}_{\rm per}(0,2\pi)$ satisfy the symmetry conditions
\begin{equation}
\label{PT-symmetry-final}
U(-x) = U(x), \quad W(-x) = -W(x), \quad \text{ for all } x\in \R.
\end{equation}
In what follows, we derive sufficient conditions for when the two real spectral bands
overlapping at a Dirac point $(k_0,\mu)$
become complex under a small perturbation. This leads to an instability of the zero solution in
the time-dependent NLS equation (\ref{nls}). At the same time Assumption \ref{ass:band} is no longer true and the formal approximation of bound states via the stationary NLS
equation (\ref{E:SNLS}) cannot be justified.

Let us first note an elementary result.

\begin{lemma}\label{L:even_odd_sol}
Fix $\gamma = 0$ and let $\mu = \omega(k_0)$ be a Dirac point of $L(k_0)$ for either $k_0 = 0$ or $k_0 = 1/2$.
The two linearly independent eigenfunctions $\phi_\pm$ of $L(k_0)$ can be chosen such that
$\phi_+(x)e^{\ri k_0 x}$ is real and even and $\phi_-(x)e^{\ri k_0 x}$ is real and odd.
\end{lemma}

\begin{proof}
At either $k_0=0$ or $k_0=1/2$ the functions $\phi_+(x)e^{\ri k_0 x}$ and $\phi_-(x)e^{\ri k_0 x}$ are
two linearly independent solutions of the Hill's equation
\beq\label{E:Hill}
\left\{ \begin{array}{l} -u''(x) + U(x)u(x)= \mu u(x), \\
u(x+2\pi)= \pm u(x) \quad \quad \text{for all } x\in\R,
\end{array} \right.
\eeq
where the plus sign is chosen for $k_0 = 0$ and the minus sign is chosen for $k_0 = 1/2$.
Since $U$ is even due to (\ref{PT-symmetry-final}),
there exists one even and one odd real-valued solution of the boundary-value problem
\eqref{E:Hill}, see Theorems 1.1 and 1.2 in \cite{MW}.
\end{proof}

The following proposition presents the first perturbation result on the unstable
splitting of Dirac points under a $\PT$-symmetric perturbation of a real even potential.

\begin{proposition}
Let the periodic potential $V$ in Assumption \ref{ass:Vsig} be given by (\ref{complex-potential}) and (\ref{PT-symmetry-final}).
Assume that $\mu = \omega(k_0)$ is a Dirac point of  $L(k_0)$ at $\gamma = 0$ for either $k_0 = 0$ or $k_0 = 1/2$
and choose the corresponding linearly independent eigenfunctions $\phi_\pm$ such that
$\phi_+(x)e^{\ri k_0 x}$ is real and even and $\phi_-(x)e^{\ri k_0 x}$ is real and odd. If
$$\langle W \varphi_+, \varphi_- \rangle\neq 0,$$
then, for every $\gamma \neq 0$ sufficiently small,
there exist two eigenvalues $\omega_\pm(k_0)$ of the spectral problem (\ref{E:Bloch-ev})
with ${\rm Im}(\omega_\pm(k_0)) \neq 0$ and $\omega_\pm(k_0)\to \mu$ as $\gamma \to 0$.
\label{prop-1}
\end{proposition}

\begin{proof}
The assertion follows from the perturbation theory for the Bloch eigenvalue problem
\begin{equation}
\label{perturbation-theory-1}
\left\{ \begin{array}{l} \left[ L_U(k_0) + \ri \gamma W \right] p(\cdot,k_0) = \omega(k_0) p(\cdot,k_0), \\
p(x+2\pi,k_0)=p(x,k_0) \quad \quad \text{for all } x\in\R,
\end{array} \right.
\end{equation}
where $L_U(k_0):=-(\tfrac{d}{dx}+\ri k_0)^2+U$ and $\mu = \omega(k_0)$ is double at $\gamma = 0$,
with two linearly independent eigenfunctions $\varphi_{\pm}$. We normalize $\varphi_{\pm}$
such that $\|\phi_\pm\|_{L^2(0,2\pi)}=1$. The eigenfunctions are orthogonal
$\langle\phi_+,\phi_-\rangle=\langle \phi_+e^{\ri k_0 \cdot},\phi_-e^{\ri k_0 \cdot}\rangle=0$
because $\phi_+e^{\ri k_0 \cdot}$ and $\phi_-e^{\ri k_0 \cdot}$ have opposite (even and odd) symmetries.

Let us use the orthogonal projection operators $P_0$ and $Q_0 = I - P_0$, such that
for every $f \in L^2(0,2\pi)$ we define
$$
P_0 f = \langle \varphi_+, f \rangle \varphi_+ + \langle \varphi_-, f \rangle \varphi_-.
$$
Then, clearly, $\langle \varphi_+, Q_0 f \rangle = \langle \varphi_-, Q_0 f \rangle = 0$.
Therefore, we write
\begin{equation}
\label{expansion-1}
\left\{ \begin{array}{l} p(\cdot,k_0) = c_+ \varphi_+ + c_- \varphi_- + \gamma \phi, \quad \langle \varphi_+, \phi \rangle = \langle \varphi_-, \phi \rangle = 0,\\
\omega(k_0) = \mu + \gamma \Omega, \end{array} \right.
\end{equation}
where $c_+,c_- \in \mathbb{C}$ are coordinates of the decomposition over the eigenfunctions
$\varphi_+$, $\varphi_-$ and $\phi$, $\Omega$ are the remainder terms (which depend on $\gamma$).
By using projection operators $P_0$ and $Q_0$,
we project the eigenvalue problem (\ref{perturbation-theory-1}) into the two blocks
\begin{equation}
\label{system-1}
\ri M_W {\bf c} + \ri \gamma \left[ \begin{array}{c} \langle \varphi_+, W \phi \rangle \\ \langle \varphi_-, W \phi \rangle \end{array} \right]
= \Omega {\bf c}
\end{equation}
and
\begin{equation}
\label{system-2}
\left[ L_U(k_0) - \mu \right] \phi = H_W := Q_0 (\Omega - \ri W) \left[ c_+ \varphi_+ + c_- \varphi_- + \gamma \phi \right],
\end{equation}
where ${\bf c} := (c_+,c_-)^T$ and
\begin{equation}
\label{matrix-M-W}
M_W := \left[ \begin{array}{cc} \langle W \varphi_+, \varphi_+ \rangle  & \langle W \varphi_-, \varphi_+ \rangle \\
\langle W \varphi_+, \varphi_- \rangle & \langle W \varphi_-, \varphi_- \rangle  \end{array} \right].
\end{equation}
Because  $W$ is odd and $|\phi_\pm(x)|^2=\phi_\pm(x)e^{\ri k_0 x}\overline{\phi_\pm(x)e^{\ri k_0 x}}$
are even, we get $\langle W \phi_\pm, \phi_\pm\rangle =0$. Hence,
$$
M_W = \left[ \begin{array}{cc} 0 & \langle W \varphi_-, \varphi_+ \rangle \\
\langle W \varphi_+, \varphi_- \rangle & 0 \end{array} \right].
$$
Since $M_W$ is hermitian, the two eigenvalues  $\Omega$
of the truncated eigenvalue problem
\begin{equation}
\label{truncated-matrix}
\ri M_W {\bf c} = \Omega {\bf c}
\end{equation}
are purely imaginary, $\Omega_{1,2}=\ri \lambda_{1,2}:= \pm \ri |\langle W \varphi_+, \varphi_- \rangle|$. They are nonzero and
distinct if $\langle W \varphi_+, \varphi_- \rangle \neq 0$.
The eigenvectors ${\bf c}$ for the two distinct eigenvalues are linearly independent.

Since the double eigenvalue $\mu$ is isolated from the rest of the spectrum of $L_U(k_0)$
in $L^2(0,2\pi)$,  a positive constant $C_0$ exists such that
$$
\| Q_0 (L_U(k_0) - \mu)^{-1} Q_0 \|_{L^2(0,2\pi) \to L^2(0,2\pi)} \leq C_0.
$$
Let us assume that ${\bf c}$ and $\Omega$ are bounded by a $\gamma$-independent positive constant
in the limit $\gamma \to 0$. Since $H_W \in {\rm Ran}(L_U(k_0) - \mu)$, fixed-point iterations can be applied to
system (\ref{system-2}) for any finite ${\bf c}$, finite $\Omega$, and sufficiently small  $\gamma > 0$.
There exists a unique solution $\phi =\phi(\gamma,\Omega,{\bf c})\in L^2(0,2\pi)$ to system (\ref{system-2})
satisfying the bound
$$
\| \phi \|_{L^2(0,2\pi)} \leq C (\| {\bf c} \| + |\Omega|),
$$
for $\gamma > 0$ sufficiently small and a $\gamma$-independent constant $C > 0$.

We substitute now  $\phi =\phi(\gamma,\Omega,{\bf c})$ into \eqref{system-1} and close the construction via an implicit function argument. Let us define
$$
G(\gamma,\Omega,{\bf c}):=\ri M_W {\bf c} +
\ri \gamma \left[ \begin{array}{c} \langle \varphi_+, W \phi \rangle \\ \langle \varphi_-, W \phi \rangle \end{array} \right]- \Omega {\bf c}.
$$
We have $G(0,\Omega_1,{\bf c}_1) = G(0,\Omega_2, {\bf c}_2) = 0$,
where $\Omega_{1,2} = \ri \lambda_{1,2}$ are the two eigenvalues of the
truncated eigenvalue problem (\ref{truncated-matrix}) with the
eigenvectors ${\bf c}_{1,2} \in \mathbb{C}^2$.
The Jacobian with respect to $\Omega$ and ${\bf c}$ is given by
$$
J_j(\tilde{\Omega},\tilde{\bf c}):=(D_{(\Omega,{\bf c})}G)|_{(0,\ri \lambda_j,{\bf c}_j)}(\tilde{\Omega},\tilde{\bf c})=\ri (M_W-\lambda_j)\tilde{\bf c}-\tilde{\Omega} {\bf c}_j, \ j=1,2.
$$

For every ${\bf b} \in \C^2$, there is a unique $\tilde{\Omega} \in \mathbb{C}$ and
$\tilde{\bf c} \in C_j^{\perp} := \{ {\bf c} \in \C^2: \;\; {\bf c} \perp {\bf c}_j\}$
such that $J_j(\tilde{\Omega},\tilde{\bf c}) = {\bf b}$. Indeed,
each ${\bf b}\in \C^2$ can be uniquely decomposed into $C_j$ and $C_j^{\perp}$
via ${\bf b} = b_j {\bf c}_j + {\bf b}^\perp_j$ for some $b_j$ and ${\bf b}^\perp_j \perp {\bf c}_j$.
Then, $\tilde{\Omega}_j = - b_j$ and $\tilde{\bf c} \in C_j^{\perp}$ is the unique solution
of the linear inhomogeneous equation $\ri (M_W-\lambda_j)\tilde{\bf c}={\bf b}^\perp_j$.

Hence, the implicit function theorem produces two unique roots for ${\bf c} \neq {\bf 0}$ and $\Omega$ in system (\ref{system-1}) which converge as $\gamma \to 0$ respectively to the eigenpairs $(\ri \lambda_1,{\bf c}_1)$ and $(\ri \lambda_2,{\bf c}_2)$  of the
truncated problem (\ref{truncated-matrix}).
\end{proof}

\begin{remark}
For a general choice of the orthogonal and normalized eigenfunctions $\varphi_+$ and $\varphi_-$,
the matrix $M_W$ in (\ref{matrix-M-W}) is no longer anti-diagonal. However, eigenvalues of $M_W$
are invariant with respect to rotation of the basis in $\mathbb{C}^2$ and therefore the two
eigenvalues are still distinct. The proof of Proposition \ref{prop-1} can be applied for a general
choice of eigenfunctions $\varphi_+$ and $\varphi_-$ and the sufficient condition
for splitting of the Dirac points is given by invertibility of the matrix $M_W$.
\end{remark}

If $U = 0$, there are infinitely many Dirac points in the Bloch eigenvalue problem (\ref{perturbation-theory-1})
for $\gamma = 0$. The following proposition gives a sufficient condition
that one of these Dirac points splits and gives rise to instability under the $\PT$-symmetric potential $W$.

\begin{proposition}
Let $U = 0$. At $\gamma=0$ Dirac points exist
at each $\mu_n=\tfrac{1}{4}n^2, n \in \N$. The $k$-coordinate of the Dirac point
at $\mu_n$ is $k_0=0$ for $n$ even and $k_0=\tfrac{1}{2}$ for $n$ odd. Let $W \in L^{\infty}_{\rm per}(0,2\pi)$
be defined by the Fourier sine series
\begin{equation}
\label{Fourier-series}
W(x) = \sum_{j \in \mathbb{N}} b_j \sin(jx)
\end{equation}
with $b_j\in \R$ for every $j \in \mathbb{N}$. If $b_n\neq 0$ for some $n\in \N$, then for every $\gamma\neq 0$
sufficiently small the Dirac point at $\mu_n = \tfrac{1}{4} n^2$ breaks into two complex eigenvalues
$\omega_\pm(k_0)$ of the spectral problem (\ref{E:Bloch-ev})
with ${\rm Im}(\omega_\pm(k_0)) \neq 0$ and $\omega_\pm(k_0)\to \mu_n$ as $\gamma \to 0$.
\label{prop-2}
\end{proposition}

\begin{proof}
For $U = 0$ and $\gamma = 0$, the eigenvalues of the Bloch eigenvalue problem \eqref{E:Bloch-ev} are
$$
\tilde{\omega}_{2m-1}(k)=(k-m)^2, \quad \tilde{\omega}_{2m}(k)=(k+m)^2, \quad m\in \N, \quad k\in \B,
$$
which give the location of the Dirac points at $(k_0,\mu)=(0,m^2)$, i.e. the crossing point of
$\tilde{\omega}_{2m-1}(k)$ and $\tilde{\omega}_{2m}(k)$, and at $(k_0,\mu)=(\tfrac{1}{2},\tfrac{(2m-1)^2}{4})$,
i.e. the crossing point of $\tilde{\omega}_{2m-2}(k)$ and $\tilde{\omega}_{2m-1}(k)$. Note that in contrast to $\omega_n$ the eigenvalues $\tilde{\omega}_n$ are not ordered according to the magnitude (of the real part) but rather according to the Fourier series index. We enumerate the Dirac points by $\mu_n := \tfrac{1}{4}n^2$ for $n \in \N$.

If $n = 2m-1$ with $m \in \mathbb{N}$, the two linearly independent normalized eigenfunctions
of the Bloch eigenvalue problem (\ref{E:Bloch-ev}) with the symmetry properties
as in Lemma \ref{L:even_odd_sol} are given by
\begin{equation}
\label{case-1}
\varphi_+(x) = \frac{1}{2 \sqrt{\pi}} \left(e^{\ri (m-1) x}+e^{-\ri mx}\right), \quad
\varphi_-(x) = \frac{1}{2 \ri \sqrt{\pi}} \left(e^{\ri (m-1) x}-e^{-\ri mx}\right).
\end{equation}
If $n = 2m$ with $m \in \mathbb{N}$, the two eigenfunctions are
\begin{equation}
\label{case-2}
\varphi_+(x) = \frac{1}{\sqrt{\pi}} \cos(mx), \quad \varphi_-(x) = \frac{1}{\sqrt{\pi}}  \sin(mx).
\end{equation}
In both (\ref{case-1}) and (\ref{case-2}) we have
$$
\langle W \varphi_+, \varphi_- \rangle = \frac{1}{2} b_n,
$$
where $b_n$ is the Fourier coefficient
in (\ref{Fourier-series}) for either $n = 2m-1$ or $n = 2m$.
If $b_n \neq 0$ for some $n \in \mathbb{N}$, the two eigenvalues $\omega_\pm(k_0)$ are complex
by Proposition \ref{prop-1}.
\end{proof}

Figure \ref{Fig:bd_str} illustrates Proposition \ref{prop-2} with the example $W(x)=\sin(2x)$
for $\gamma = 0$ (a) and $\gamma = 0.2$ (b,c).
The eigenvalue families in (b,c) were computed numerically using a finite difference discretization.

\begin{remark}
If $b_j\neq 0$ for all $j\in \N$ in the Fourier series \eqref{Fourier-series},
then all Dirac points split into two complex eigenvalues and none of the spectral bands is completely real for $\gamma$ small.
\end{remark}

\begin{figure}[ht!]
\includegraphics[scale=0.5]{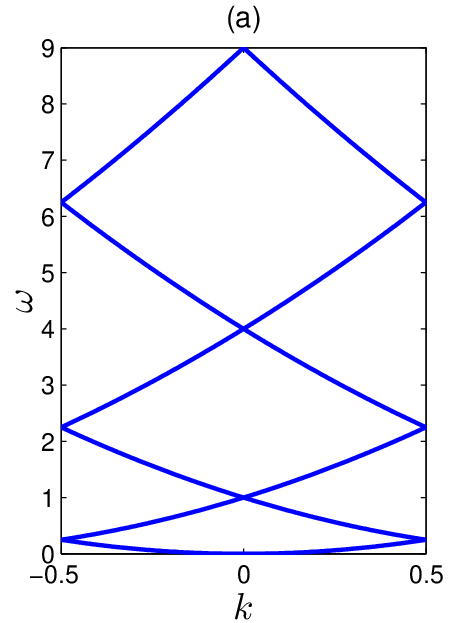}
\includegraphics[scale=0.5]{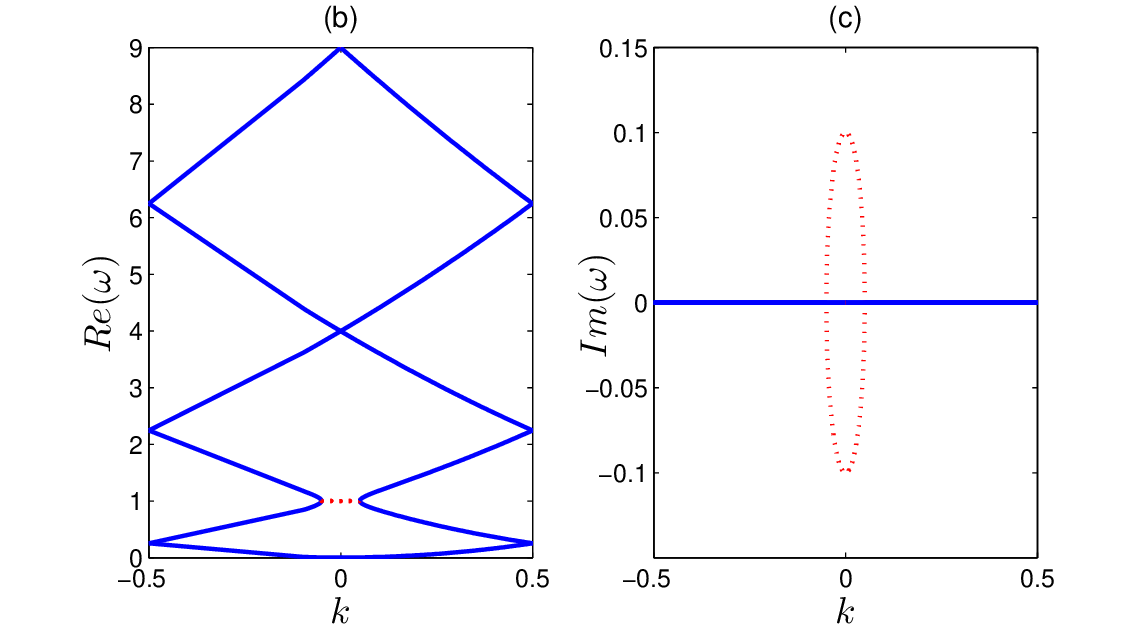}
\caption{Eigenvalues $\omega_n(k)$, $n=1,\dots,6$ of the Bloch eigenvalue problem \eqref{E:Bloch-ev}
with $V\equiv 0$ in (a) and $V(x)=0.2\ri \sin(2x)$ in (b) and (c) computed numerically. The Dirac point at the intersection of $\tilde{\omega}_2$ and $\tilde{\omega}_3$ with $V\equiv 0$ splits into a complex conjugate pair when $V(x)=0.2\ri \sin(2x)$. Purely real eigenvalues are plotted with the full blue line, complex eigenvalues are in dotted red.}\label{Fig:bd_str}
\end{figure}

The final result shows that if $U$ is smoother than $W$ and $W$ is not too smooth, then
the high-energy bands split generally and become unstable for every nonzero $\gamma$.
This means that the $\PT$-symmetry breaking threshold discussed in many
publications (see, e.g., the review in \cite{Konotop}) is identically zero
even if the real potential $U$ is generic and has no Dirac points.
To simplify the proof of the following proposition, we assume that $U$ has zero mean.

\begin{proposition}
Let $U$ and $W$ be defined by the Fourier series
\begin{equation}
\label{Fourier-series-2}
U(x) = 2 \sum_{j \in \mathbb{N}} a_j \cos(jx), \quad
W(x) = 2 \sum_{j \in \mathbb{N}} b_j \sin(jx),
\end{equation}
where $\{a_j\}_{j \in \mathbb{N}}, \{b_j\}_{j \in \mathbb{N}} \in \ell^1(\mathbb{N},\mathbb{R})$ satisfy
\begin{equation}
\label{assumption-a-b}
\lim\limits_{j\to\infty} \frac{|a_j|}{|b_j|} = 0, \quad
\lim\limits_{j \to \infty} \frac{1}{j^2 |b_j|} = 0, \quad
\lim\limits_{j \to \infty} \frac{\sum_{k = j+1}^{\infty} |b_k|^2}{j^2 |b_j|^2} = 0.
\end{equation}
Then for every $\gamma \neq 0$ there is a sufficiently large $N \in \mathbb{N}$
such that for every $n \geq N$ two complex eigenvalues $\omega_{\pm (n)}(k_0)$
of the spectral problem (\ref{E:Bloch-ev}) with  $k_0\in \{0, 1/2\}$ exist, which satisfy
\begin{itemize}
\item[(i)] ${\rm Im}(\omega_{\pm (n)}(k_0)) \neq 0$,
\item[(ii)] $|\omega_{\pm (n)}(k_0) - n^2/4| \to 0$ as $n \to \infty$.
\end{itemize}
\label{prop-3}
\end{proposition}

\begin{proof}
By the asymptotic theory in \cite[Chapter 4]{Eastham} for $\gamma = 0$,
the band edge points converge
at infinity to the Dirac points of the homogenous problem (\ref{E:Bloch-ev}) with $V = 0$.
Therefore, in order to prove the assertion, we will
treat $U$ and $W$ as perturbation terms in the Bloch eigenvalue problem
\begin{equation}
\label{perturbation-theory-2}
\left\{ \begin{array}{l} \left[ L_0(k_0) + U + \ri \gamma W \right] p(\cdot,k_0) = \omega(k_0) p(\cdot,k_0), \\
p(x+2\pi,k_0)=p(x,k_0) \quad \quad \text{for all } x\in\R,
\end{array} \right.
\end{equation}
where $L_0(k_0):=-(\tfrac{d}{dx}+\ri k_0)^2$. The two eigenfunctions
of $L_0(k_0)$ are given by either (\ref{case-1}) or (\ref{case-2}) for the double
eigenvalue $\omega(k_0) = \frac{1}{4} n^2$ with either $n = 2m-1$ and $k_0 = 1/2$
or $n = 2m$ and $k_0 = 0$.

We present here the even case $n = 2m$, $m \in \N$. The odd case is analogous.
We represent
$$\omega(k_0) = m^2 + \Omega,$$
where $\Omega$ is shown to be small as $m \to \infty$. Let us write $p(\cdot,k_0)$ in the Fourier series form
$$
p(x,k_0) = \sum_{j \in \mathbb{Z}} c_j e^{i j x}.
$$
Substituting these representations in the Bloch eigenvalue problem (\ref{perturbation-theory-2}),
we obtain the discrete eigenvalue problem
\begin{equation}
\label{perturbation-theory-3}
(j^2 - m^2 - \Omega) c_j + \sum_{k \in \mathbb{Z}} (a_k + \gamma b_k) c_{j-k} = 0,
\end{equation}
where $a_{-k} = a_k$ and $b_{-k} = -b_k$ for $k \in \mathbb{N}$ and $a_0 = b_0 = 0$.
Singling out the resonant terms at $j = \pm m$, we project the eigenvalue problem (\ref{perturbation-theory-3})
into two blocks
\begin{equation}
\label{system-1a}
\left[ \begin{array}{cc} 0 & a_m + \gamma b_m \\
a_m - \gamma b_m & 0 \end{array} \right] {\bf C}_m +
\left[ \begin{array}{c} \sum_{j \in \mathbb{Z} \backslash \{0,2m\}} (a_j + \gamma b_j) c_{m-j} \\
\sum_{j \in \mathbb{Z} \backslash \{0,-2m\}} (a_j + \gamma b_j) c_{-m-j} \end{array} \right]
= \Omega {\bf C}_m
\end{equation}
and
\begin{equation}
\label{system-2a}
c_j = \frac{1}{m^2 + \Omega - j^2} \sum_{k \in \mathbb{Z}} (a_k + \gamma b_k) c_{j-k}, \quad j \in \mathbb{Z} \backslash \{ m, -m\},
\end{equation}
where ${\bf C}_m := (c_m,c_{-m})^T$. The two eigenvalues of the matrix in the first term of
the left-hand side of (\ref{system-1a}) are given by
\begin{equation}
\label{eig-truncated}
\Omega^{(0)}_{\pm} := \pm \sqrt{a_m^2 - \gamma^2 b_m^2}.
\end{equation}
Since $\lim\limits_{m\to\infty}|a_m|/|b_m| = 0$ by the first assumption in (\ref{assumption-a-b}),
for any $\gamma \neq 0$ there exists a sufficiently large $N$ such that
$a_m^2 - \gamma^2 b_m^2 < 0$ for any $m \geq N$. The corresponding
eigenvalues $\Omega^{(0)}_{\pm}$ are distinct and complex.
In what follows, we prove persistence of this complex splitting
of the double zero eigenvalue $\Omega$ in the block (\ref{system-1a}).

We assume that ${\bf C}_m$ and $\Omega$ are bounded by an $m$-independent positive constant
in the limit $m \to \infty$. We denote ${\bf c}_m := \{ c_j \}_{j \in \mathbb{Z}_m}$
with $\mathbb{Z}_m := \mathbb{Z} \backslash \{ m,-m\}$ and work in the sequence space
$\ell^2(\mathbb{Z}_m)$, which represents the space $L^2(0,2\pi)$ for
the original problem (\ref{perturbation-theory-2}).

Since the spacing between $m^2$ and $(m\pm 1)^2$ grows
like $m$ as $m \to \infty$, we set
$$
K_j := \frac{1}{m^2 + \Omega - j^2}\sum_{k\in\mathbb{Z}\backslash\{j-m,j+m\}} (a_k+\gamma b_k)c_{j-k}, \quad j \in \mathbb{Z}_m
$$
and obtain
\begin{equation}
\label{inverse-operator}
\left\| K \right\|_{\ell^2(\mathbb{Z}_m)}
\leq C m^{-1} \left(\| a \|_{\ell^1} + |\gamma| \| b \|_{\ell^1}\right) \| {\bf c}_m \|_{\ell^2(\mathbb{Z}_m)},
\end{equation}
by using  Young's inequality for convolutions
\begin{equation}
\label{Young}
\left\| x * y \right\|_{\ell^r}\leq \|x\|_{\ell^p}\|y\|_{\ell^q}, \quad r,p,q \geq 1, \quad  1+\frac{1}{r} = \frac{1}{p} + \frac{1}{q},
\end{equation}
with $r=2$, $p=1$, and $q=2$. The positive constant $C$ is $m$-independent but may depend
on $\Omega$ and $\gamma$. In what follows,
we use the same notation for the generic constant $C$ that may change from one line to another line.

Thanks to the bound (\ref{inverse-operator}), the inverse operator can be constructed for
system (\ref{system-2a}) in $\ell^2(\mathbb{Z}_m)$ for any finite ${\bf C}_m$ and $\Omega$ if $m$ is sufficiently large.
By the inverse function theorem, there exists a unique solution
${\bf c}_m \in \ell^2(\mathbb{Z}_m)$ to system (\ref{system-2a}), which can be represented in the form
\begin{equation}
\label{c-representation}
{\bf c}_m = c_{m} {\bf P}_m(\Omega,\gamma) + c_{-m} {\bf Q}_m(\Omega,\gamma),
\end{equation}
where the unique vectors ${\bf P}_m, {\bf Q}_m \in \ell^2(\mathbb{Z}_m)$ depend on $m$, $\Omega$, and $\gamma$
and satisfy the bounds
\begin{equation}
\label{bounds-P-Q}
\| {\bf P}_m(\Omega,\gamma) \|_{\ell^2(\mathbb{Z}_m)} + \| {\bf Q}_m(\Omega,\gamma) \|_{\ell^2(\mathbb{Z}_m)} \leq
C m^{-1}
\end{equation}
for an $m$-independent positive constant $C$.

By the symmetry of system (\ref{system-2a}), we note that
\begin{equation}
\label{P-Q-symmetry}
\left[ {\bf Q}_m(\Omega,\gamma) \right]_{-j} = \left[ {\bf P}_m(\Omega,-\gamma) \right]_{j}, \quad j \in \mathbb{Z}_m.
\end{equation}
Moreover, solving system (\ref{system-2a}) by iterations, we can write
\begin{equation}
\label{P-leading-order}
\left[ {\bf P}_m(\Omega,\gamma) \right]_{j} = \frac{a_{j-m} + \gamma b_{j-m}}{m^2 + \Omega - j^2} +
\left[ \tilde{\bf P}_m(\Omega,\gamma) \right]_{j}, \quad
j \in \mathbb{Z}_m.
\end{equation}
where $\tilde{\bf P}_m$ satisfies the system
$$
\left[ \tilde{\bf P}_m(\Omega,\gamma) \right]_{j} = (m^2 + \Omega - j^2)^{-1}\sum_{k\in \mathbb{Z}_m \backslash \{j-m,j+m\}}(a_k+\gamma b_k)
\left[{\bf P}_m(\Omega,\gamma)\right]_{j-k}, \quad j \in \mathbb{Z}_m.
$$
Thanks to Young's inequality (\ref{Young}), the higher order terms satisfy the bound
\begin{equation}
\label{bounds-higher-order-terms}
\| \tilde{\bf P} \|_{\ell^2(\mathbb{Z}_m)} \leq  Cm^{-1} \|{\bf P}_m \|_{\ell^2(\mathbb{Z}_m)}\leq C m^{-2},
\end{equation}
for another $m$-independent positive constant $C$.

Substituting (\ref{c-representation}) to (\ref{system-1a}), we obtain the matrix nonlinear eigenvalue problem in the form
\begin{equation}
\label{system-3}
\left[ \begin{array}{cc} E_m(\Omega,\gamma) &
a_m + \gamma b_m + F_m(\Omega,\gamma) \\
a_m - \gamma b_m + G_m(\Omega,\gamma)  & H_m(\Omega,\gamma) \end{array} \right] {\bf C}_m
= \Omega {\bf C}_m,
\end{equation}
where
\begin{eqnarray*}
E_m(\Omega,\gamma) & := & \sum_{j \in \mathbb{Z} \backslash \{0,2m\}} (a_j + \gamma b_j) [ {\bf P}_m(\Omega,\gamma)]_{m-j}, \\
F_m(\Omega,\gamma) & := & \sum_{j \in \mathbb{Z} \backslash \{0,2m\}} (a_j + \gamma b_j) [ {\bf Q}_m(\Omega,\gamma)]_{m-j}, \\
G_m(\Omega,\gamma) & := & \sum_{j \in \mathbb{Z} \backslash \{0,-2m\}} (a_j + \gamma b_j) [ {\bf P}_m(\Omega,\gamma)]_{-m-j}, \\
H_m(\Omega,\gamma) & := & \sum_{j \in \mathbb{Z} \backslash \{0,-2m\}} (a_j + \gamma b_j) [ {\bf Q}_m(\Omega,\gamma)]_{-m-j}.
\end{eqnarray*}
By the symmetry in (\ref{P-Q-symmetry}), we obtain
$$
E_m(\Omega,\gamma) = H_m(\Omega,-\gamma), \quad F_m(\Omega,\gamma) = G_m(\Omega,-\gamma).
$$
Eigenvalues $\Omega$ are found as roots of the characteristic equation for (\ref{system-3}),
namely
\begin{equation}
\label{roots-Omega}
\left[ \Omega - E_m^+(\Omega,\gamma) \right]^2 = \left[ a_m + F_m^+(\Omega,\gamma) \right]^2
- \left[ \gamma b_m + F_m^-(\Omega,\gamma) \right]^2  + \left[ E_m^-(\Omega,\gamma) \right]^2
\end{equation}
where $E_m^{\pm}$ and $F_m^{\pm}$  define the symmetric and anti-symmetric combinations of $E_m$ and $F_m$ respectively, e.g.
$$
E_m^{\pm}(\Omega,\gamma) := \frac{E_m(\Omega,\gamma) \pm E_m(\Omega,-\gamma)}{2}.
$$
Substituting the leading order (\ref{P-leading-order}), we find
\begin{eqnarray*}
E_m(\Omega,\gamma) & = & \sum_{j \in \mathbb{Z} \backslash \{0,2m\}} \frac{a_j^2 - \gamma^2 b_j^2}{m^2 + \Omega - (m-j)^2} + {\rm h.o.t.}, \\
F_m(\Omega,\gamma) & = & \sum_{j \in \mathbb{Z} \backslash \{0,2m\}} \frac{(a_j+\gamma b_j)(a_{j-2m}-\gamma b_{j-2m})}{m^2 + \Omega - (m-j)^2} + {\rm h.o.t.},
\end{eqnarray*}
where the higher-order terms are convolutions of ${\bf a}+\gamma {\bf b}$ and $\tilde{\bf P}$ estimated in (\ref{bounds-higher-order-terms}). Hence, by Young's inequality, ${\rm h.o.t.}$ is bounded in the $\ell^\infty$ norm by $C/m^2$.

Since the leading order of $E_m(\Omega,\gamma)$ is even in $\gamma$, we obtain the estimates:
\begin{eqnarray*}
|E_m^+(\Omega,\gamma)| & \leq & C_1 m^{-1}, \\
|E_m^-(\Omega,\gamma)| & \leq & C_2 |\gamma| m^{-2},
\end{eqnarray*}
for $m$-independent constants $C_1, C_2$, where the factor of $\gamma$ is included for convenience.
On the other hand, the leading order of $F_m(\Omega,\gamma)$ can be estimated as follows:
\begin{eqnarray*}
\left| \sum_{j \in \mathbb{Z} \backslash \{0,2m\}} \frac{a_j a_{j-2m}}{m^2 + \Omega - (m-j)^2} \right| & = &
\left| \frac{a_m^2}{m^2+\Omega} + 2 \sum_{j \geq m+1, j \neq 2m} \frac{a_j a_{j-2m}}{m^2 + \Omega - (m-j)^2} \right| \\
& \leq & C \left( \frac{a_m^2}{m^2} + \frac{ (\sum_{j\in \N} |a_j|^2)^{1/2} (\sum_{j \geq m+1} |a_j|^2)^{1/2}}{m} \right),
\end{eqnarray*}
with similar estimates for the other parts of the leading order of $F_m(\Omega,\gamma)$.
Combining with the higher-order terms and recalling the first assumption in (\ref{assumption-a-b}),
we obtain the estimates
\begin{eqnarray*}
|F_m^+(\Omega,\gamma)| & \leq & C_3 m^{-1} \left(\sum_{j \geq m+1} |b_j|^2 \right)^{1/2} + C_4 m^{-2}, \\
|F_m^-(\Omega,\gamma)| & \leq & C_5 |\gamma| m^{-1} \left(\sum_{j \geq m+1} |b_j|^2\right)^{1/2} + C_6 |\gamma| m^{-2},
\end{eqnarray*}
for $m$-independent constants $C_3, C_4, C_5, C_6$.

The right-hand side of the nonlinear characteristic equation (\ref{roots-Omega})
is
$$
R_m:=b_m^2\left(-\gamma^2 + \frac{a_m^2}{b_m^2}+2\frac{a_mF_m^+}{b_m^2}+\frac{(F_m^+)^2}{b_m^2}-2\frac{\gamma F_m^-}{b_m}
-\frac{(F_m^-)^2}{b_m^2}+\frac{(E_m^-)^2}{b_m^2}\right).
$$
Using the three assumptions in (\ref{assumption-a-b}), we conclude that $R_m = b_m^2 (-\gamma^2 + \delta_m)$,
where $|\delta_m| \to 0$ as $m\to \infty$. Therefore, $R_m < 0$ if $m$ is sufficiently large.

We note next that the roots $\Omega$ of the characteristic equation (\ref{roots-Omega}) are bounded in $m$ as
they are fixed points of
$$\Omega=E_m^+(\Omega,\gamma)\pm \sqrt{R_m(\Omega,\gamma)},$$
where $|E_m^+(\Omega,\gamma)|$ is estimated above and $|\sqrt{R_m(\Omega,\gamma)}|\leq C|b_m|(|\gamma|+\sqrt{|\delta_m|)}$. Hence
$$
|\Omega|\leq C_7 m^{-1}
$$
for some $C_7>0$ independent of $m$. This leads to an estimate on the imaginary part of $E_m^+$. Namely, since
$${\rm Im}(E_m^+)=-{\rm Im}(\Omega)\left(\sum_{j \in \mathbb{Z} \backslash \{0,2m\}} \frac{a_j^2 - \gamma^2 b_j^2}{(m^2 + {\rm Re}(\Omega) - (m-j)^2)^2+{\rm Im}(\Omega)^2} + {\rm h.o.t.}\right),$$
we have
$$|{\rm Im}(E_m^+)| \leq C/m^3.$$
Finally, thanks to the second assumption in \eqref{assumption-a-b},
the imaginary part of the two roots of $\Omega$ is nonzero if $m$ is so large that $|\delta_m|/\gamma^2<1$ because then the following asymptotics hold
$${\rm Im}(\Omega) \sim \pm |\gamma b_m| \sqrt{1-\delta_m/\gamma^2} \quad (m\to \infty).$$
The assertion of the proposition is thus proved.
Note that $\Omega_{\pm}^{(0)}$ given by (\ref{eig-truncated}) may be smaller
than the leading order term for $\Omega$ given by $E_m^+ = \mathcal{O}(m^{-1})$.
\end{proof}

As an example for the assumptions in Proposition \ref{prop-3}, we consider
the periodic potentials $U$ and $W$ such that
$$
\begin{aligned}
&|a_m| = \mathcal{O}(m^{-5/2}) \quad \mbox{\rm for }  m\to \infty, \\
\quad C_- m^{-3/2}\leq &|b_m| \leq C_+ m^{-3/2} \quad \mbox{\rm for all } m \; \mbox{\rm sufficiently large}
\end{aligned}
$$
with some $0 < C_- \leq C_+ < \infty$.
Since $\sum_{j \geq m+1} |b_j|^2 = \mathcal{O}(m^{-2})$, the assumptions in (\ref{assumption-a-b}) are
satisfied and by Proposition \ref{prop-3},
for every $\gamma > 0$, there exists $N$ such that
eigenvalues (\ref{eig-truncated}) are complex for every $m \geq N$.
Moreover, there is a positive constant $C$ such that
$N \geq C \gamma^{-2}$. The latter estimate follows from
$|\delta_m| \leq c(m^{-1} + \gamma^2 m^{-1/2})$ obtained
from the previous estimates on $E_m^{\pm}$ and $F_m^{\pm}$ as well as
the definition of $\delta_m$. If $C$ is sufficiently large in $N \geq C \gamma^{-2}$,
then $|\delta_m|/\gamma^2$ is small for every $m \geq N$.

\end{document}